\documentclass[12pt]{article}
\usepackage{amssymb,amsmath,amsfonts,amsthm,enumerate,color}
\usepackage[toc,page,title,titletoc,header]{appendix}
\usepackage{graphicx}
\usepackage{savesym}
\usepackage{mathrsfs}
\usepackage{subfigure}
\usepackage[latin1]{inputenc}
\usepackage[T1]{fontenc}
\usepackage{indentfirst}
\usepackage{eurosym}
\usepackage{multicol}
\usepackage{booktabs}
\usepackage{hyperref}
\usepackage[T1]{fontenc}

\numberwithin{equation}{section}

\newtheorem{Theorem}{Theorem}[section]
\newtheorem{Lemma}[Theorem]{Lemma}
\newtheorem{Example}[Theorem]{Example}

\newtheorem{Proposition}[Theorem]{Proposition}
\newtheorem{Definition}[Theorem]{Definition}
\newtheorem{Corollary}[Theorem]{Corollary}

\newtheorem{Remark}[Theorem]{Remark}
\numberwithin{equation}{section}

\def\Rs{\mathcal{R}}

\def\pf{\mathfrak{p}}
\def\Pr{\mathscr{P}}
\def\Qr{\mathscr{Q}}
\def\T{\mathcal{T}}
\def\embed{\hookrightarrow}
\def\la{\langle}
\def\ra{\rangle}

\def\wto{\rightharpoonup}

\def\C{{\mathbb C}}

\def\M{{\mathcal{M}}}
\def\xdag{{x^\dag}}

\def\Lom3{L^2(\Omega)^3}

 \def\p{\partial} 
\def \Vh0{\stackrel{\circ}{V}_h} \def\to{\rightarrow}
   
\def\Om{\Omega}  
  
    \def\R{{\mathbb R}}
  \def\f{\frac}  

\def\p{\partial}

\newcommand{\lc}
{\mathrel{\raise2pt\hbox{${\mathop<\limits_{\raise1pt\hbox
{\mbox{$\sim$}}}}$}}}

\newcommand{\gc}
{\mathrel{\raise2pt\hbox{${\mathop>\limits_{\raise1pt\hbox{\mbox{$\sim$}}}}$}}}

\newcommand{\ec}
{\mathrel{\raise2pt\hbox{${\mathop=\limits_{\raise1pt\hbox{\mbox{$\sim$}}}}$}}}

\def\bb{\begin{equation}} \def\ee{\end{equation}}

\def\beqn{\begin{eqnarray}}  \def\eqn{\end{eqnarray}}
\def\beq{\begin{equation}} \def\eeq{\end{equation}}

\def\beqnx{\begin{eqnarray*}} \def\eqnx{\end{eqnarray*}}

\def\bn{\begin{enumerate}} \def\en{\end{enumerate}}

\def\bd{\begin{description}} \def\ed{\end{description}}



\title{\bf Variational source conditions in $L^p$-spaces}
\author{
De-Han Chen\footnote{ School of Mathematics \& Statistics, Central China Normal University,   Wuhan 430079, China
(chen.dehan@uni-due.de). The work of DC was financially supported by
National Natural Science Foundation of China
(Nos.  11701205 and 11871240).  }
\and Irwin Yousept \footnote{Universit\"{a}t Duisburg-Essen, Fakult\"{a}t f\"{u}r Mathematik, Thea-Leymann-Str. 9, D-45127 Essen, Germany,
 (irwin.yousept@uni-due.de). The work of IY was financially supported   by the German Research Foundation Priority Programm DFG
SPP 1962 "Non-smooth and Complementarity-based Distributed Parameter Systems: Simulation and Hierarchical Optimization",  Project YO 159/2-2.
} \footnote{Author to whom any correspondence should be addressed.}
}
\begin{document}

\date{}
\maketitle

\begin{abstract}
We propose and analyze variational source conditions (VSC) for the Tikhonov regularization method with  $L^p$-penalties applied to an ill-posed operator equation in a Banach space. Our analysis   is built on    the celebrated Littlewood-Paley theory and the  concept of   (Rademacher) R-boundedness.   With these two analytical principles, we validate the proposed VSC under         a conditional stability estimate in terms of a dual Triebel-Lizorkin-type norm.  In the final part of the paper,   the developed theory is applied to   an inverse elliptic problem with measure data for the reconstruction of
possibly unbounded diffusion coefficients.  By means of VSC,    convergence rates for the associated Tikhonov regularization with  $L^p$-penalties are obtained.
\end{abstract}



\section{Introduction}
Let us consider an ill-posed  operator equation  of the type
\beq\label{op}
T(x)=y  \quad \textrm{in } Y,
\eeq
where $Y$ is a Banach space, and   $T:D(T)  \to Y$ is a weakly sequentially continuous mapping with a closed and convex subset $D(T) \subset L^p(\Omega,\mu)$ for some  $1<p<+\infty$  and $\sigma$-finite measure $(\Omega,\mu)$.  We underline that the Lebesgue space  $L^p(\Omega,\mu)$ is  real, but the Banach space $Y$ is allowed to be complex or real.  Moreover, the right hand side   $y$ lies in the range of  $T$.  The operator equation   \eqref{op} is supposed  to be {\it locally ill-posed} with a specific characterization  through the stability estimate  \eqref{eq:cond}  comprising a   dual    Triebel-Lizorkin-type norm. The local degree of the ill-posedness is exactly described by the involved dual norm in  \eqref{eq:cond} (Remark \ref{remfortheo} (ii)). To construct a stable approximation to the ill-posed problem \eqref{op}, we employ the celebrated Tikhonov regularization method  taking into account  the noisy data  $y^\delta \in Y$  under the deterministic noise model: $\|y-y^\delta\|_Y\leq \delta$.  More precisely, for a given    $\alpha>0$,  the   solution of  \eqref{op}  is approximated by a minimizer     of
 \beq\label{Tik}
 \min_{x\in D(T)} T_\alpha^\delta(x):= \f{1}{\ell}\|T(x)-y^\delta\|^\ell_Y+\f{\alpha}{p}\|x-x^*\|_{p}^{\hat p},
  \eeq
for a fixed constant $\ell> 1$, $\hat{p}:=\max\{p,2\}$, and a fixed a priori guess $x^*$ of $x$.  In view of the presupposed conditions on $T:D(T)  \to Y$ and $D(T) \subset L^p(\Omega,\mu)$, the
existence and plain convergence   for the Tikhonov regularization method  \eqref{Tik} follow by standard arguments  (see \cite{EHN96,hofmann2007convergence,schuster2012regularization}).

In general, a convergence rate  for \eqref{Tik}  is guaranteed   under    a  smoothness assumption  on the true solution, well-known as the so-called {\it source condition} (cf. \cite{EHN96,Eng89,hofmann2007convergence,Kaltenbacher08,Kaltenbacher09}).  However,  classical source conditions are rather restrictive since they require  the  Fr\'{e}chet differentiability of the operator $T$ and further properties on its first-order derivative (see \cite{Chavent94,Eng89,EZ00,FJZ12, Kaltenbacher08,Kaltenbacher09,Lechleiter08,SEK93}). These restrictions make  the classical source condition less appealing for the convergence analysis of \eqref{Tik}. Our  focus is therefore set on the   concept of     {\it variational source condition} (VSC)
introduced originally by Hofmann et al. \cite{hofmann2007convergence}  in the case of  a linear index function. Convergence rates based on VSC for a general index function  were shown independently in \cite{boct2010extension,flemming2010theory,grasmair2010generalized}.  In contrast to the classical source condition,    VSC is applicable to a wider class of inverse problems with possibly non-smooth forward operators.  More importantly, convergence rates can be deduced from VSC in a straightforward manner  (cf. Hofmann and Math\'e \cite{HofMat12})  without any additional nonlinearity condition such as tangential cone condition.  We refer to
   Hohage and Weidling \cite{hohage2017characterizations} for a general characterization  of VSC in Hilbert spaces.     See also \cite{chen2018variational,hohage2015verification,weidling2015variational} regarding   VSC  for inverse problems governed by partial differential equations (PDEs).  All these results were derived by means of the spectral theory for self-adjoint operators in Hilbert spaces.

Although the study of VSC was initiated in
the Banach space setting,  general    sufficient conditions   for VSC in Banach spaces are somewhat  restrictive (see   \cite{flemming2018theory2,schuster2012regularization}), compared with  those for  the Hilbertian case,  which are mainly related to conditional stability estimates and smoothness of the true solution. Such methodology have been applied to various inverse problems governed by PDEs in the Hilbertian setting {(see \cite{chen2018variational,hohage2015verification,hohage2017characterizations,weidling2015variational})}.  More recently,   less restrictive sufficient conditions for VSC in Besov spaces were proposed
by  Hohage et al. \cite{hohage2019optimal,hohage2020} using a new characterization of subgradient smoothness. Their results lead to optimal convergence rates for the Tikhonov regularization method with wavelet Besov-norm penalties.   However, we note that $ L^p(\Omega,\mu)$ for $p\neq 2$ is not  a Besov space, and therefore \cite{hohage2019optimal,hohage2020} are not directly applicable to \eqref{op}-\eqref{Tik}.

Striving to fill   this gap, our paper develops   novel sufficient criteria for  VSC in the  $L^p$-setting based on a sophisticated  application of
 the Littlewood-Paley decomposition and the concept of the (Rademacher) $\mathcal{R}$-boundedness. The Littlewood-Paley theory is a systematic method to understand various properties of functions  by decomposing them in infinite dydic sums with frequency localized components.     On the other hand, the concept of $\mathcal{R}$-boundedness was initially introduced to  study multiplier theorems for vector-valued functions \cite{bourgain1986vector}. These two mathematical concepts are of central significance in the vector-valued harmonic analysis and its profound application to PDEs (cf.  \cite{bourgain1986vector,hytonen2018analysis}). For the sake of completeness, we provide some basics and standard results concerning the Littlewood-Paley decomposition and $\mathcal{R}$-boundedness in Sections \ref{LPD} and \ref{RB}. Invoking these two analytical tools, we   prove our main result (Theorem \ref{the:vsc}) on the sufficient criteria for VSC in the $L^p$-setting, leading to convergence rates for the Tikhonov regularization method \eqref{Tik}.
The proposed sufficient conditions consist of the existence of a Littlewood-Paley decomposition for the (complex)  space   $L^q(\Omega,\mu;\C)$,  $q:= \frac{p}{p-1}$, together with the previously mentioned conditional stability estimate  \eqref{eq:cond}
and a regularity assumption for the true solution  in terms of  a  Triebel-Lizorkin-type  norm.

The final part of this paper focuses on an inverse reconstruction problem of     possibly  unbounded diffusion $L^p$-coefficients in elliptic equations with measure data. Such   problems are mainly motivated from geological or medical  applications involving dirac measures as source terms. They include  acoustic  monopoles  in full waveform inversion (FWI) and  electrostatic phenomena  with a current dipole source  in Electroencephalography (EEG).  We analyze the mathematical property of the corresponding forward operator  and prove the existence and plain convergence of the corresponding regularized solution (Theorem \ref{tikop}). Finally, we transfer our abstract theoretical finding  to this specific inverse problem and verify its requirements (see Theorem \ref{first} and Lemmas \ref{lemma:prior} and \ref{lemma:prior2}), leading to  convergence rates for the associated Tikhonov regularization method (Corollary \ref{corfinal}).

\section{Preliminaries}

We begin by recalling some terminologies and notations used in the sequel.  Let $X,Y$ be   complex or real Banach spaces.   The space of all linear and bounded operators from $X$ to $Y$ is denoted by
$
B(X,Y)= \{A: X \to Y \textrm{ is     linear and bounded}\},
$
endowed with  the operator norm $\|A\|_{B(X,Y)}:= \sup_{\|x\|_X=1} \|Ax\|_{Y}$.
 If $X=Y$, then we simply write $B(X)$ for $B(X,X)$. The notation $X^*$ stands for the dual space of $X$.   A linear  operator $A:D(A) \subset X  \to X$
is called closed, if its graph $\{(x,Ax), ~ x \in D(A)\}$ is closed in  $X\times X$. If $A:D(A) \subset X  \to X$ is a linear and closed   operator,  then
$$
\rho(A) := \{\lambda \in \C   \mid    \lambda\textup{id} - A : D(A) \to X \textrm{ is  bijective and}\,
(\lambda\textup{id} - A)^{-1}\in B(X) \}
$$
and
$$
\sigma(A) := \C \setminus \rho(A)
$$
denote respectively  the resolvent set and spectrum of $A$. For every $\lambda \in \rho(A)$, the operator
$
R(\lambda,A):= (\lambda \textup{id} - A)^{-1} \in B(X)
$
is referred to as the resolvent operator of $A$.

      If $(\Omega,\mu)$ is a  $\sigma$-finite measure and $1\leq p<+\infty$, then $L^p(\Omega,\mu)$ (resp. $L^p(\Omega,\mu; \C)$) denotes the space of all equivalence classes of  real-valued  (resp. complex-valued) $\mu$-measurable and $p$-integrable  functions with the corresponding norm $\|f\|_{p}:= \left(\int_\Omega |f|^p d\mu\right)^{\frac{1}{p}}$. If $\Omega \subset \R^n$ is a measurable  and $\mu$ is the Lebesgue measure,   then we simply write $L^p(\Omega)$ (resp. ${L^q(\Omega;\C)}$)  for $L^p(\Omega,\mu)$ (resp. ${L^q(\Omega,\mu;\C)}$).  For $f,g\in L^1(\R^n;\C)$, $f\star g$ denotes the convolution of $f$ and $g$.
Moreover, let  $ \langle \cdot,\cdot\rangle_{p,q}:=\int_{\Omega} f\overline{g}d\mu $ stand for the duality product between $f\in L^p(\Omega,\mu;\C)$ and  $g\in L^q(\Omega,\mu;\C)$ for $\f{1}{p}+\f{1}{q}=1$.

Finally,  for nonnegative real numbers $a,b$, we write $a\lesssim b$, if
$a \leq C   b$ holds true for  a positive constant $C>0$ independent of  $a$ and $b$. If $a\lesssim b$ and $b\lesssim a$, we then write $a\cong b$.

\subsection{Sobolev spaces}
For every $-\infty< s<\infty$ and $p \geq 1$, we define the (classical) fractional Sobolev space
$$
H^s_p(\R^n;\C):=\{u\in \mathcal{S}(\R^n;\C)'\mid \|u\|_{H^s_p(\R^n;\C)}:= \|\mathcal{F}^{-1}[(1+|\cdot |^2)^{\f{s}{2}}|(\mathcal{F}u)]\|_{L^p(\R^n;\C)}  <+\infty  \},
$$
where  $\mathcal{S}(\R^n;\C)'$ denotes the tempered   distribution space and  $\mathcal{F}:\mathcal{S}(\R^n;\C)'\to \mathcal{S}(\R^n;\C)'$ is the Fourier transform (see, e.g.,  \cite{Yagi}). For a bounded open set $ U \subset \R^n$ with a Lipschitz boundary $\p U$,  the space  $H^s_p(U;\C)$ with a possibly non-integer exponent $s\geq 0$ is defined as the space of all  complex-valued functions $v\in L^p(U;\C)$ satisfying
$V_{\vert U}=v$ for some $V\in H^s_p(\R^n;\C)$, endowed with the norm
$$
\|v\|_{H^s_p(U;\C)}:=\inf_{\substack{V_{\vert U}=v \\ V\in H_p^s(\R^n;\C)}}\|V\|_{H^s_p(\R^n;\C)}.
$$
 Furthermore, the real   counterpart to  $H^{s}_p(U;\C)$ is simply denoted by $H^{s}_p(U)$.

\begin{Proposition}[{\cite{{taylor2007tools}}, \cite[Theorem 4.10.1]{triebel1995interpolation} and \cite[Theorem 1.36]{Yagi}}]\label{pro:sobolev}
Let $\Omega$ be a bounded  open set  in $\R^n$ with a Lipschitz boundary and $1<\tau<+\infty$.

\begin{enumerate}

\item[\textup{(i)}]    If $\tau<n$, then for any $1\leq s\leq \f{\tau n}{n-\tau}$, the embedding  $H^1_\tau(\Omega;\C)\embed L^s(\Omega;\C)$ is continuous.  It is compact if $s<\f{\tau n}{n-\tau}$.

\item[\textup{(ii)}] If $\tau \ge n$, then for any $1\leq s<+\infty$, the embedding  $H^1_\tau (\Omega;\C)\embed L^s(\Omega;\C)$ is compact.

\item[\textup{(iii)}] Let $0\le s_1,s_2<+\infty$ and $1\leq \tau_1,\tau_2\leq +\infty$. Furthermore,  let $\rho\in (0,1)$ and
\begin{align*}
s:= (1-\rho) s_1+\rho s_2,\quad
\f{1}{\tau}:= \f{1-\rho}{\tau_1} + \f{\rho}{\tau_2}.
\end{align*}
Then, there exists a constant $C>0$ such that
\beq\label{interpolate}
\|u\|_{H^s_\tau(\Omega;\C)} \le C \|u\|_{H^{s_1}_{\tau_1}(\Omega;\C)}^{1-\rho} \|u\|^{\rho}_{H^{s_2}_{\tau_2}(\Omega;\C)} \quad \forall \, u\in H^{s_1}_{\tau_1}(\Omega;\C)\cap H^{s_2}_{\tau_2}(\Omega;\C).
\eeq

\end{enumerate}

\end{Proposition}

In the following, we also summarize the well-known composition rule and product estimates for Sobolev functions (cf. \cite[Chapter 2, Propositions 1.1 and 6.1]{taylor2007tools}).

\begin{Proposition}[Composition rule and product estimates]\label{pro:chain}
	Let $\Omega$ be a bounded open  set in $\R^n$ with a Lipschitz boundary. 	
	\begin{enumerate}
		\item[\textup{(i)}] Let $1\leq \tau <+\infty$. If $F:\mathbb{C}\to \mathbb{C}$ is  globally Lipschitz  and satisfies $F(0)=0$, then  $F(u)\in H^1_\tau(\Omega;\C)$ holds  true for all $u\in H^1_\tau(\Omega;\C)$.
		
		\item[\textup{(ii)}] For all $1<\tau,  \tau_1, \tau_2<+\infty$ satisfying  $\f{1}{\tau}=\f{1}{\tau_1}+\f{1}{\tau_2}$,	
		 there exists a constant $C>0$ such that
		$$
		\|uv\|_{H^1_\tau(\Omega;\C)}\leq C\|u\|_{H^1_{\tau_1}(\Omega;\C)}\|v\|_{H^1_{\tau_2}(\Omega;\C)}	
		$$
		holds true for all   $u\in H^1_{\tau_1}(\Om;\C)$ and $ v\in  H^1_{\tau_2}(\Omega;\C)$.
	\end{enumerate}
\end{Proposition}

\subsection{Littlewood-Paley decomposition} \label{LPD}

In its simplest manifestation, the Littlewood-Paley  (LP) decomposition is a method to understand various properties of functions by decomposing them into an infinite dydic sums of frequency localized components.   A prominent  example  for an LP decomposition can be found in the classical theory of harmonic analysis as follows: Let $1<q<+\infty$ and $s\geq 0$. Then, every $f\in H^s_q(\R^n;\C)$ can be decomposed into
\beq\label{clp}
f= \sum_{j=0}^\infty  f\star  \check{\varphi}_j  \ \ \text{and}\,\ \
\|f\|_{H^s_q(\R^n;\C)}\cong \|(\sum_{j=0}^\infty 2^{2js} |f\star  \check{\varphi}_j|^2)^{\f{1}{2}} \|_{q},
\eeq
where $\{\varphi_j\}_{j=0}^\infty$ is a family of compactly supported smooth functions satisfying $\text{supp}(\varphi_0) \subset \{\xi\mid |\xi|\leq 2\}$,  $\text{supp}(\varphi_1)\subset \{\xi\mid 1\leq |\xi|\leq 4\}$, $\varphi_j(\cdot):=\varphi_1(\cdot 2^{1-j})$ for $j\geq 2$ and
$
\sum_{j=0}^\infty \varphi(\xi)=1
$ for all $\xi\in \R^n$. Furthermore,  $ \check{\varphi}_j$ denotes the inverse Fourier transfomation of $ {\varphi}_j$ (cf. \cite[Section 4.1]{stein1993harmonic}). Motivated by \eqref{clp}   and following \cite{kriegler2016paley}, we introduce the following definition:

\begin{Definition}\label{def:LP}
Let $(\Omega,\mu)$ be a $\sigma$-finite measure and $1<q<+\infty$.
We say that  $L^q(\Omega,\mu;\C)$ admits  an LP decomposition  if there is a family of uniformly bounded and {pairwisely commutative} linear operators  $\{P_j\}_{j=0}^\infty\subset B( L^q(\Omega,\mu;\C))$ satisfying the following conditions:
\begin{enumerate}
\item[\textup{(i)}] The partition of identity:
\beq\label{partition}
z=\sum_{j=0}^\infty P_j z    \quad \forall\, z\in  L^q(\Omega,\mu;\C).
\eeq
\item[\textup{(ii)}]  Almost orthogonality:
\beq\label{almost}
P_jP_k z=0  \quad \forall\, z\in L^q(\Omega,\mu;\C) \quad \forall j,k\in \mathbb N \cup \{0\} \textrm{ with }   |j-k|\geq 2.
\eeq
\item[\textup{(iii)}] Norm equivalence: there exists a constant  $c^*\geq 1$ such that
\beq\label{decomp:A}
\frac{1}{c^*}  \|z\|_{q}  \leq  \|(\sum_{j=0}^\infty |P_j z|^2)^{\f{1}{2}}\|_{q}   \leq c^*   \|z\|_{q}   \quad \forall z\in  L^q(\Omega,\mu;\C).
\eeq
\end{enumerate}
\end{Definition}
\begin{Remark}    \normalfont The third condition in Definition \ref{def:LP} implies that  $\{P_j\}_{j=0}^\infty$ is uniformly bounded in $B( L^q(\Omega,\mu;\C))$.  Therefore, we may  remove the uniform boundedness assumption on  $\{P_j\}_{j=0}^\infty$ in the definition.
From the partition of identity  and the almost orthogonality, it follows that
\beq\label{pj}
P_j(P_{j}+P_{j-1}+P_{j+1})=P_j\quad \forall j\geq 1.
\eeq
Note that \eqref{clp} gives a classical example of  an LP decomposition on $L^q(\R^n;\C)$. Also,
if $q=2$ and $\{e_j\}_{j=0}^\infty$ is an orthonormal basis of $L^2(\Omega,\mu;\C)$, then the family of operators $\mathcal{P}=\{P_j\}_{j=0}^\infty$ with  $P_jz:=(z ,e_j)_{L^2(\Omega,\mu;\C)} z$ is an LP decomposition on $L^2(\R^n;\C)$.
\end{Remark}

With the help of the  LP decomposition and inspired by the classical Triebel-Lizorkin spaces, if $L^q(\Omega,\mu)$ admits an LP decomposition $\mathcal{P}=\{P_j\}_{j=0}^\infty\subset B( L^q(\Omega,\mu;\C))$, then
  the following space
\beq\label{Fspace}
F^s_{q}(\mathcal{P}):=\{z\in L^q(\Omega,\mu;\C)\mid \|z\|_{F^s_{q}(\mathcal{P})}:=\|(\sum_{j=0}^\infty 2^{2js}|P_j z|^2)^{\f{1}{2}}\|_{q}<+\infty\},\quad \forall s\geq 0,
\eeq
defines a  Banach space. Obviously, $F^0_q(\mathcal{P})= L^q(\Omega,\mu;\C)$ holds true with norm equivalence.  According to Definition \ref{def:LP}, $F^s_q(\mathcal{P})$ is  a dense subspace of $L^q(\Omega,\mu;\C)$, and the embedding $F^s_q(\mathcal{P})\embed L^q(\Omega,\mu;\C) $ is  continuous.

\subsection{$\Rs$-boundedness} \label{RB}

\begin{Definition}\label{def:R}
Let $(\Omega,\mu)$ be a $\sigma$-finite measure and $1<q<+\infty$.
  A subset $\T\subset  B( L^q(\Omega,\mu;\C))$ is called $\Rs$-bounded, if there exists a constant $C>0$  such that for all $n\in \mathbb{N}$, $T_1,\ldots, T_n \in \T $ and $z_1,\ldots, z_n \in  L^q(\Omega,\mu;\C) $, the following inequality holds
\beq\label{def:R0}
\|(\sum_{k=1}^n |T_k z_k|^2)^{\f 1 2}\|_{q}\leq  C  \|(\sum_{k=1}^n |z_k|^2)^{\f 1 2}\|_{q}.
\eeq
The  infimum of all such constants $C>0$ is called the $\Rs$-bound of $\T$ and denoted by $\Rs(\T)$.
\end{Definition}

\begin{Remark}    \normalfont
The notion of $\Rs$-boundedness  can also be  defined  by  using Rademacher functions, and  Definition   \ref{def:R} is also referred to as $\ell^2$-boundedness (cf. \cite{hytonen2018analysis}) { or $\mathcal{R}^2$-boundedness (cf. \cite{Kunstmann}).}  By Khintchine's inequality, these two definitions are equivalent in $ L^q(\Omega,\mu;\C)$ (see \cite[Remark 4.1.3]{pruss2016moving} or \cite[Proposition 6.3.3]{hytonen2018analysis}).  Since  our work only focuses    on   $L^q(\Omega,\mu;\C)$ and considers   \eqref{def:R0},  we  choose the terminology  ``$\Rs$-boundedness''.  We note that  for the case $q=2$,  the $\Rs$-boundedness of $\T$ is equivalent to the  uniform boundedness of $\T$ (cf.  \cite[Remark 4.1.3]{pruss2016moving}).

\end{Remark}

We recall some elementary properties regarding to $\Rs$-boundedness.

\begin{Proposition}[cf. {\cite[Example 8.1.7 and Proposition 8.1.19]{hytonen2018analysis} and \cite[Propositions 2.9 and 2.10]{Kunstmann}}]\label{pro:R} Let $(\Omega,\mu)$ be a $\sigma$-finite measure and $1<q<+\infty$.  Then, the following claims hold true:
\begin{enumerate}
\item[\textup{(i)}]  Every singleton $\{T\}$  in $ B(L^q(\Omega,\mu;\C))$ is $\Rs$-bounded with
$$
\Rs(\{T\})\leq C_G\|T\|_{B(L^q(\Omega,\mu;\C))},
$$
where $C_G>0$ denotes the Grothendieck's constant.  In particular,
$\mathcal{R}(\{\textup{id}\})=1$.

\item[\textup{(ii)}]  If $\mathcal{T}, \mathcal{S}\subset B(L^q(\Omega,\mu;\C))$ are $\Rs$-bounded subsets, then both
$\mathcal{T}+ \mathcal{S}$ and $\mathcal{T}\cup \mathcal{S}$ are $\Rs$-bounded with
$$
\Rs(\mathcal{T}+ \mathcal{S})\leq \Rs(\mathcal{T})+\Rs(\mathcal{S})\quad \text{and} \quad \Rs(\mathcal{T}\cup \mathcal{S})\leq \Rs(\mathcal{T})+\Rs(\mathcal{S}).
$$
\end{enumerate}

\end{Proposition}
 Let us mention that the exact value of Grothendieck's constant is still an open problem, and it is known that $\f{\pi}{2}\leq C_G\leq \f{\pi}{2\ln(1+\sqrt{2})}$ (cf. \cite{krivine1978constantes}).  A direct consequence of Proposition \ref{pro:R} is summarized in the following corollary:

\begin{Corollary}
If a subset $\mathcal{T}\subset B( L^q(\Omega,\mu;\C))$ is finite, then it is $\Rs$-bounded.
\end{Corollary}

\subsection{Existence of  LP decompositions via sectorial operators}

In this section, we recall the notion of the sectorial operator and  discuss some LP decomposition for $L^q(\Omega,\mu;\C)$ with the help of sectorial operators. In the following, let $X$ be a complex Banach space.
 For $\omega\in (0,\pi)$, let
$\Sigma_\omega:=\{z\in \mathbb{C}\backslash\{0\}\mid |\arg z|<\omega\}$ denote the symmetric sector around the positive axis of aperture angle $2\omega$.

\begin{Definition}[\cite{Hasse,kriegler2016paley}]
Let $\omega\in (0,\pi)$.  A  linear and closed operator $A:D(A)\subset X\to X$ is called $\omega$-sectorial if the following conditions hold:
\begin{enumerate}
 \item[\textup{(i)}] the   spectrum $\sigma(A)$ is contained in $\overline{\Sigma_\omega}$;

 \item[\textup{(ii)}] $R(A)$ is dense in $X$;

 \item[\textup{(iii)}]
$
\forall\, \theta\in (\omega,\pi) \,\, \exists \,\, C_\theta>0 \,\, \forall \,\lambda\in \mathbb{C}\backslash \overline{\Sigma_\theta}:
  \|\lambda R(\lambda,A)\|\leq C_\theta.
$
\end{enumerate}
We say that $A$ is $0$-sectorial operator, if $A$ is $\omega$-sectorial for all  $\omega \in (0,\pi)$.
\end{Definition}
 Note that (ii) and (iii) imply that every $\omega$-sectorial operator is injective (cf. \cite{Hasse}).
For every $\theta\in (0,\pi)$, we denote by $H^\infty(\Sigma_\theta; \C)$  the space of all bounded holomorphic functions on $\Sigma_\theta$, which is a Banach algebra with the norm $\|f\|_{\infty,\theta}:=\sup_{z\in \Sigma_\theta}|f(z)|$.
Moreover, we introduce   the subspace
$H_0^\infty(\Sigma_\theta;\C):=\{f\in H^\infty(\Sigma_\theta;\C)\mid \exists\, C, \epsilon>0 \,\text{such that}\, |f(z)|\leq C \f{|z|^\epsilon}{(1+|z|)^{\epsilon}}\}$. Then, for an $\omega$-sectorial operator $A$ and a function $f\in H_0^\infty(\Sigma_\theta;\C)$ with $\theta\in (\omega,\pi)$, one can define a   linear and bounded operator
\beq\label{Cauchy}
 G_A(f):X\to X,\quad G_A(f):=\f{1}{2\pi i}\int_{\Gamma} f(\lambda)R(\lambda,A) d\lambda,
\eeq
where $\Gamma$ is the boundary of the sector $\Sigma_\sigma$ with $\sigma\in (\omega,\theta)$, oriented counterclockwise.  Note that by the Cauchy integral formula  for vector-valued holomorphic functions, the above integral has the same value for all $\sigma \in (\omega,\theta)$. Therefore, the definition \eqref{Cauchy} is independent of the choice of $\Gamma$.
If there exists a constant $C>0$ such that
 $$\|G_A(f)\|_{B(X)}\leq C\|f\|_{\infty,\theta}\quad \forall  f\in H^\infty_0(\Sigma_\theta;\C),
 $$  then we say that $A$ has a bounded $H^\infty_0(\Sigma_\theta;\C)$-calculus. In this case, the Cauchy integral formula \eqref{Cauchy} can be extended to a bounded
homomorphism $H^\infty(\Sigma_\theta;\C)\to B(X), \,f\mapsto G_A(f)$.   For any $\alpha\geq 0$, we can choose an integer $n$ strictly larger than
$\alpha$ such that the function $f_\alpha(z):=z^\alpha(1+z)^{-n}$ belongs to $ H_0^\infty(\Sigma_\theta;\C)$, and so
  the operator
$$
A^\alpha:D(A^\alpha)\subset X\to X, \quad
A^\alpha :=(1+A)^{n}G_A(f_\alpha)
 $$
 defines a linear and closed operator (cf. \cite[Lemma A.1.3]{Hasse}) with the effective domain
 $
 D(A^\alpha):=\{x\in X\mid  G_A({f_\alpha}) x\in D(A^n)\}.
 $
In particular,  $D(A^\alpha)$ equipped with the graph norm
\beq\label{interp}
 \|A^\alpha \cdot\|_X+\|\cdot\|_X
\eeq
defines a Banach space.  Clearly, $D(A^0)=X$ and $D(A^1)=D(A)$.

Let $(A,D(A))$ be a 0-sectorial operator and $\eta>0$. If there is a constant $C>0$  such that for all $\omega\in (0,\pi)$,
$$
\|G_A(f)\|_{B(X)}\leq \f{C}{\omega^\eta}\|f\|_{\infty,\omega} \quad \forall \, f\in H^\infty(\Sigma_\omega),
$$
then we say that $A$ has  a (bounded) $\mathcal{M}^\eta$-calculus (see e.g. \cite[Theorem 4.10]{cowling1996banach} and \cite{kriegler2016paley}).  Another equivalent definition of  $\mathcal{M}^\eta$-calculus
can be found in \cite{kriegler2016paley} (see  \cite[Theorem 4.1]{cowling1996banach} for the proof).

Under the existence of a 0-sectorial operator with   $\mathcal{M}^\eta$-calculus for some $\eta>0$, the following key lemma guarantees
the existence of an LP decomposition for $L^q(\Omega,\mu;\C)$:


\begin{Lemma}[{\cite[Theorem 4.1 and Theorem 4.5]{kriegler2016paley}}]\label{lemma:m}
Let $(\Omega,\mu)$ be a $\sigma$-finite measure.
 If $X=L^q(\Omega,\mu;\C)$  for some
$1<q<+\infty$, and  there exists a $0$-sectorial operator $A:D(A)\subset X\to X$ with $\M^\eta$-calculus for some $\eta>0$,  then $X$ admits an LP-decomposition $\mathcal{P}=\{P_j\}_{j=0}^\infty$  such that
\beq\label{equvi:F}
F^s_q(\mathcal{P})=D(A^s)\quad \forall \,s\geq  0,
\eeq
where $F^s_q(\mathcal{P})$ is defined as in  \eqref{Fspace}.
\end{Lemma}

\begin{Example}\label{ex:PL}

 Let $\Omega$ be a bounded domain of $\R^n$ ($n\geq 2$) with a
$C^{1,1}$-boundary and $X=L^q(\Omega;\C)$ for $1<q<+\infty$.

\begin{enumerate}

\item[\textup{(i)}]{\bf Dirichlet boundary condition}.   If we define
$A u:=-\Delta u $ for all $u\in D(A )$ with $D(A):=H^2(\Omega;\C)\cap \mathring{H^1}(\Omega;\C)$, which corresponds to Dirichlet boundary condition, then $A: D(A)\subset L^2(\Omega;\C)\to L^2(\Omega;\C)$ is a self-adjoint operator with $0\in \rho(A)$ and {$-A$} generates a strongly continuous semigroup $\{e^{-At}\}_{t\geq 0}$, whose kernel $\{p_t\}_{t\in (0,+\infty)}$ satisfies the following Gaussian upper bound estimate:
\beq\label{heat}
|p_t(x,y)|\lesssim \f{1}{t^{\f{n}{2}}}\exp(-c\f{|x-y|^2}{t})\quad \forall (t,x,y)\in (0,+\infty)\times \Omega\times \Omega,
\eeq
 for some $c>0$
 (see e.g. \cite[Theorem 6.10]{ouhabaz2009analysis} and \cite[Chapter 7]{ouhabaz2009analysis}).    We can extend  $\{e^{-At}\}_{t\geq 0}$ to  a strongly continuous semigroup on $L^q(\Omega;\C)$, whose  generator is denoted by $-A_{q}:D(A_{q})\subset L^q(\Omega;\C)\to L^q(\Omega;\C)$. Then, $A_{q}:D(A_{q})\subset L^q(\Omega;\C)\to L^q(\Omega;\C)$ is a 0-sectorial operator with bounded $\M^\eta$-calculus  for $\eta>\lfloor \f{n}{2}\rfloor+1 $ (see e.g. \cite[Lemma 6.1]{kriegler2016paley}  and \cite[Theorem 7.23]{ouhabaz2009analysis}), where $\lfloor c\rfloor$ denotes the largest integer smaller  than $c$.
   Therefore,  according to Lemma \ref{lemma:m}, $L^q(\Omega;\C)$ admits an LP decomposition $\mathcal{P}_D$  and
$$
F_q^\theta(\mathcal{P}_D)= D(A^\theta_{q})=\begin{cases} H^{2\theta}_q(\Omega;\C)\, &0\leq\theta<\f{1}{2q},\\
\{H^{2\theta}_q(\Omega;\C)\mid \gamma u=0\}\, &1\geq \theta>\f{1}{2q}\,\text{and}\,\theta\neq  \f{q+1}{2q},
\end{cases}
$$
 where  we have used the characterization of $D(A_{q}^\theta)$ from  \cite[Theorem 16.15]{Yagi}.

\item[\textup{(ii)}]{\bf Neumann boundary condition}. Let us now consider  $A u:=-\Delta u+u$ for $u\in D(A)$, where
 $D(A):=\{ H^2(\Omega;\C)\mid \f{\p u}{\p n}=0 \,\text{on}\, \partial\Omega\}$.
 Then, $A:D(A)\subset L^2(\Omega;\C)\to L^2(\Omega;\C)$ is also a self-adjoint operator with $0\in \rho(A)$ and hence, { $-A$}  generates a strongly continuous semigroup $(e^{-At})_{t\geq 0}$.  Its kernel also satisfies the  classical Gaussian upper estimate \eqref{heat} (\cite[Theorem 6.10]{ouhabaz2009analysis} and \cite[Chapter 7]{ouhabaz2009analysis}).  As in the first case, $(e^{-At})_{t\geq 0}$ can be extended to a strongly continuous semigroup on
 $L^q(\Omega;\C)$ with the generator denoted by
$-A_{q}:D(A_{q})\subset L^q(\Omega;\C)\to L^q(\Omega;\C)$.  Again, thanks to \cite[Lemma 6.1]{kriegler2016paley}  and \cite[Theorem 7.23]{ouhabaz2009analysis},  $A_{q}: D(A_{q})\subset L^q(\Omega;\C)\to L^q(\Omega;\C)$ is  a 0-sectorial operator over $L^q(\Omega;\C)$ with bounded $\M^\eta$-calculus  for  $\eta>\lfloor \f{n}{2}\rfloor+1 $. Therefore, by Lemma  \ref{lemma:m},  $L^q(\Omega;\C)$  admits an LP
 decomposition $\mathcal{P}_N$  such that
\beq\label{projection:N}
F_q^\theta(\mathcal{P}_N)=D(A^\theta_{q})=\begin{cases} H^{2\theta}_q(\Omega;\C)\ &0\leq\theta<\f{q+1}{2q},\\
\{H^{2\theta}_q(\Omega;\C)\mid \f{\p u}{\p n}=0\}\ &1\geq \theta>\f{q+1}{2q},
\end{cases}
\eeq
where  we have used the characterization of $D(A_{q}^\theta)$ from \cite[Theorem 16.11]{Yagi}.
\end{enumerate}
\end{Example}
Further examples     for sectorial operators   with bounded $\M^\eta$-calculus can be found in  \cite[Lemma 6.1]{kriegler2016paley}.

\section{Sufficient conditions for VSC in $L^p(\Omega,\mu)$}

In all what follows, let $(\Omega,\mu)$ be a $\sigma$-finite measure,     $1<p<+\infty$, $q:=\f{p}{p-1}$, and $\hat p := \max \{p,2\}$.  It is well-known that the real  Lebesgue space  $L^p(\Omega, \mu)$ is $\hat p$-uniformly convex (see e.g. \cite{xu1991inequalities}),  and  there exists a constant $c_p>0$  such that
\beq\label{eq:uni}
\|w+y\|_{p}^{\hat p}\geq \|w\|^{\hat p}_{p}+\hat p \langle y, J_{\hat p} (w)\rangle_{p,q}+c_p\|y\|^{\hat p}_{p} \quad \forall w,y \in  L^p(\Omega, \mu),
\eeq
 where $J_{\hat p} : L^p(\Omega, \mu) \to L^q(\Omega, \mu)$ denotes the generalized duality map  (cf. \cite{xu1991inequalities}) satisfying
\begin{equation}\label{general}
\la w,J_{\hat p}(w) \ra_{p,q}=\|w\|_p^{\hat p}\quad \textup{and}\quad \|J_{\hat p}(w)\|_q=\|w\|^{\hat{p}-1}_p.
\end{equation}
 Given  a $x^*$-minimum norm  solution $x^\dag \in D(T) \subset L^p(\Omega,\mu)$  to the ill-posed operator equation \eqref{op}, i.e.,
$$
\|x^\dag-x^*\|_{p}=\min\{\|x-x^*\|_{p}\mid x\in D(T)\,\text{such that} \, T(x)=y\},
$$
our goal is to find a constant $\beta\in (0,c_{p})$ and a concave index function $\Psi:(0,+\infty)\to (0,+\infty)$  such that  the following VSC
 \beq\label{VSCp}
 \langle x^\dag-x, J_{\hat p}(x^\dag-x^*) \rangle_{p,q}
\leq \f{c_{p}-\beta}{\hat p}\|x-\xdag\|^{\hat{p}}_{p}+\Psi(\|T(x)-T(\xdag)\|_Y)\quad \forall\,x\in D(T)
 \eeq
 holds true.  Note that  a function $\Psi:(0,+\infty)\to(0,+\infty)$ is called an index function if it is continuous, strictly increasing and satisfies the limit condition $\lim_{\delta\to 0^+}\Psi(\delta)=0$.

 \begin{Remark}   \normalfont Inserting $y= x-\xdag$ and $w=x^\dag-x^*$ in   \eqref{eq:uni}, we immediately obtain that
 $$
\langle x^\dag-x,  J_{\hat p}(x^\dag-x^*) \rangle_{p,q}\geq \f{1}{\hat p}
\left(\|\xdag-x^*\|^{\hat p}_{p}-\|x-x^*\|^{\hat p}_{p}+c_p\|x-\xdag\|_{p}^{\hat p}\right)   \quad  \forall x\in D(T).
$$
Therefore,  \eqref{VSCp} implies that
\begin{equation}\label{VSC}
\f{\beta}{\hat p}\|x^\dag-x\|^{\hat p}_p\leq \f{1}{\hat p}\|x-x^*\|_p^{\hat p}-\f{1}{\hat p}\|x^\dag-x^*\|^{\hat p}_p+\Psi(\|T(x)-T(x^\dag)\|_Y)\quad \forall x\in {D(T)}.
\end{equation}


 \end{Remark}

 VSC of the   type  \eqref{VSC} has been proposed in \cite{HofMat12,schuster2012regularization}. Thus, as \eqref{VSCp} implies   \eqref{VSC},
 the following convergence rate  result  follows directly from  \cite[Theorem 1]{HofMat12} and \cite[Theorem 4.13]{schuster2012regularization}):

 \begin{Corollary} \label{corconv}
Suppose that VSC \eqref{VSCp} holds true for some $\beta\in (0,c_{p})$ and  concave  index function $\Psi:(0,+\infty)\to (0,+\infty)$. If
the regularization parameter   in \eqref{Tik} is chosen as  $\alpha(\delta):=\f{\delta^\ell}{\Psi(\delta)}$, then every solution $x^\delta_{\alpha(\delta)} \in D(T)$ to \eqref{Tik}  satisfies
\beq\label{eq:regular:conv}
\|x^\delta_{\alpha(\delta)}-x^\dag\|_p^{\hat p}=O(\Psi(\delta))\quad \text{as}\,\,\delta\to 0^+.
\eeq
\end{Corollary}

Let us now state our main assumption on the existence of an LP decomposition for the dual space  of $L^p(\Omega, \mu;\C)$:

\begin{enumerate}

\item[{\bf (H0)}]  $L^q(\Omega,\mu;\C)$ admits an LP decomposition $\mathcal{P}=\{P_j\}_{j=0}^\infty\subset B( L^q(\Omega,\mu;\C))$ in the sense of  Definition \ref{def:LP}.
\end{enumerate}

If {\bf (H0)} holds, then for every $\theta\geq 0$, we can construct a Banach space $F^{\theta}_q:=F^{\theta}_q(\mathcal{P})$  by
\eqref{Fspace}.  Since the embedding $F^\theta_q\embed L^q(\Omega,\mu;\C)$ is dense and continuous,   the embedding $L^p(\Omega,\mu;\C)\embed (F^\theta_{q})^*$ is continuous, and therefore
\beq\label{eq:pair}
|\la f,g\ra_{p,q}| \le   \|f\|_{(F^\theta_q)^*}\|g\|_{F^\theta_q}\quad \forall \, (f,g)\in  L^p(\Omega, \mu; \C) \times F^\theta_q.
\eeq

\begin{Theorem}\label{the:vsc} Let $(\Omega,\mu)$ be a $\sigma$-finite measure,     $1<p<+\infty$, and $q=\f{p}{p-1}$ satisfying
  {\bf (H0)}. Suppose that   there exist a concave index function $\Psi_0: (0,+\infty)\to (0,+\infty)$ and a constant  $\theta\geq 0$ such that
\beq\label{eq:cond}
\|\xdag-x\|_{(F^\theta_{q})^*} \lesssim \Psi_0(\|T(\xdag)-T(x)\|_Y) \quad \forall \, x\in D(T).
\eeq
Moreover, assume that  $f^\dag:=J_{\hat p}(\xdag-x^*)$ is nonzero and belongs to $F^{s\theta}_{q}$  for some $0< s\leq 1$. Then, VSC \eqref{VSCp} holds true for  $\beta = \frac{c_p}{2}$ and a   concave index function $\Psi:(0,+\infty)\to (0,+\infty)$, defined by
\begin{equation}\label{def:psitheo}
\Psi(\delta):=
\begin{cases}
C\|f^\dag\|_{  F_q^\theta }\Psi_0(\delta) \quad \text{if}\, \,s=1,\\
\displaystyle
C\inf_{\lambda\geq 1}\left[\f{1}{2^{\lambda \hat q s}}\|f^\dag\|_{F^{s\theta}_q}^{\hat q}  +2^{(\lambda+1)(1-s) }\|f^\dag\|_{F^{s\theta}_{q}}\Psi_0(\delta)\right]  \, \text{if}\, \,s\in (0,1),
\end{cases}
\end{equation}
for all sufficiently large      $C>0$ and $\hat{q}:=\min\{q,2\}$. Furthermore, the index function \eqref{def:psitheo} satisfies
\beq\label{psidecay}
\Psi(\delta)\lesssim   \Psi_0(\delta)^{\f{\hat qs}{1+(\hat q-1)s}}
\eeq
 for all sufficiently small $\delta>0$.
\end{Theorem}
\begin{Remark}    \normalfont
\label{remfortheo}

~\begin{enumerate}

\item[\textup{(i)}]  If $f^\dag$ is zero, then $x^*=x^\dag$, i.e.,   the a priori guess $x^*$ is exactly the  true solution $x^\dag$. In this case, VSC \eqref{VSCp} holds true for all $\beta\in (0,c_p)$ and all index functions $\Psi$.

\item[\textup{(ii)}]  The condition \eqref{eq:cond} characterizes the local ill-posedness of
the forward operator $T:L^p(\Omega,\mu) \supset D(T)  \to Y$ at $\xdag$. As the topology of $(F^\theta_q)^*$ becomes  coarser  for growing $\theta$, i.e., $(F^{\theta_1}_q)^* \subset (F^{\theta_2}_q)^* $ holds for any $0 \le \theta_1 < \theta_2$,  the ill-posedness grows if $\theta$ becomes larger. On the other hand,   if $\theta=0$, then $(F^{0}_q)^*= L^p(\Omega,\mu;\C)$, and \eqref{eq:cond} implies the local well-posedness at $\xdag$ in the  following sense:
\[
 \{x_n\}_{n\in \mathbb{N}}\subset D(T) \quad \textrm{and} \quad  \lim_{n\to \infty}T(x_n) = T(\xdag)  \text{ in }  Y \implies  \lim_{n\to \infty }x_n =  \xdag   \text{ in }  L^p(\Omega,\mu).
\]
\vspace{-0.6cm}
\item[\textup{(iii)}]
The existence of a concave index function $\Psi_0$ satisfying
 \eqref{eq:cond}  can be obtained by conditional stability estimates, including H\"{o}lder/Lipschitz-type estimates and logarithmic type estimates, for the corresponding inverse problem \eqref{op} related to the forward operator $T:  L^p(\Omega,\mu) \supset D(T)   \to Y$.   The claim for the case of $\theta=0$ can be found in \cite[Theorem 4.26]{schuster2012regularization}.  In this case,  the assumption {\bf (H0)} is not required, and    \eqref{VSCp} holds   for all  $\beta\in (0,c_p)$ and  $\Psi=C \|f^\dag\|_{  F_q^\theta } \Psi_0$ for all sufficiently large   $C>0$.

  \item[\textup{(iv)}] We underline that the regularity condition $f^\dag:=J_{\hat p}(\xdag-x^*)\in F^{s\theta}_{q}$     is not a source condition. This regularity requirement along with  {\bf (H0)} and   the stability estimate \eqref{eq:cond} yield that  VSC \eqref{VSCp} is satisfied for the index function \eqref{def:psitheo}. No other smoothness conditions are needed.

 \end{enumerate}
\end{Remark}
\begin{proof}

If $s=1$ or $\theta=0$, then     \eqref{eq:pair} and \eqref{eq:cond} imply
that
\beq\label{case0}
\la \xdag-x, J_{\hat p}(x^\dag-x^*) \ra_{p,q}\le     \|\xdag-x\|_{ (F^\theta_{q})^* }\|f^\dag\|_{  F^\theta_{q}}\le  \|f^\dag\|_{  F_q^\theta }  \Psi_0(\|T(\xdag)-T(x)\|_Y)
\eeq
holds true  for all $x\in D(T)$. Therefore, if $s=1$ or $\theta=0$, VSC \eqref{VSCp} is satisfied for all $\beta \in (0,c_p)$ and   $\Psi(\delta)= \|f^\dag\|_{  F_q^\theta }\Psi_0(\delta)$.

We now prove the claim for  $0<s<1$ and $\theta>0$.  To this aim, let   $x\in D(T)$ be arbitrarily fixed.
For any fixed $\lambda\geq 1$, we introduce
$$  \mathscr{P}_{\lambda} z:=\sum_{k = 0}^{\lfloor \lambda \rfloor} P_k z \quad \forall z \in L^q(\Omega, \mu) \quad  \textrm{and} \quad
\mathscr{Q}_{\lambda}:=I-\Pr_{\lambda},
$$
where we recall that $\lfloor \lambda \rfloor \in \mathbb N$ denotes the largest integer satisfying $\lfloor \lambda \rfloor \le \lambda$. Then,
\beq\label{decomp}
\la \xdag-x, f^\dag \ra_{p,q}
=\la \xdag-x, \Qr_\lambda f^\dag \ra_{p,q}+\la \xdag-x, \Pr_\lambda f^\dag \ra_{p,q}=:{\bf I_1}+{\bf I_2}.
\eeq
Let us  first derive a proper  estimate for     ${\bf I_1}$.  Since $\hat{p}=\max\{2,p\}$ and $\hat{q}=\min\{q,2\} =\f{\hat{p}}{\hat{p}-1}$,  Young's inequality implies that
\begin{align}\label{est0}
{\bf I_1}&\leq \| \xdag-x\|_{p}\|\Qr_\lambda f^\dag\|_{q}
\leq \f{c_p}{2\hat{p}} \| \xdag-x\|_{p}^{\hat{p}}+ \f{1}{\hat q}\left(\f{2}{c_p}\right)^{\hat q-1}\|\Qr_\lambda f^\dag\|_{q}^{\hat{q}}.
\end{align}
 Next, in view of   the almost orthogonality \eqref{almost} and  the   partition of identity \eqref{partition}, it holds for all $z\in L^q(\Omega,\mu)$ that
 \beq\label{proj0}
  P_j\Qr_{\lambda } z=
  P_j  \! \!\! \!\sum_{k= \lfloor \lambda \rfloor+1}^{\infty}  \! \!\! \! P_k z= \begin{cases}
 P_j z,\quad  j\geq { \lfloor \lambda \rfloor+2},\\
P_{ \lfloor \lambda \rfloor+1} (P_{ \lfloor \lambda \rfloor+1}+P_{ \lfloor \lambda \rfloor+2}) z  \quad j= \lfloor \lambda \rfloor+1,\\
P_{ \lfloor \lambda \rfloor} P_{ \lfloor \lambda \rfloor+1}z\quad j= \lfloor \lambda \rfloor,\\
 0,\quad  j\leq { \lfloor \lambda \rfloor}-1.
 \end{cases}
 \eeq
By \eqref{decomp:A}, \eqref{proj0}, and the fact that $\{P_j\}_{j=0}^\infty$ is pairwisely commutative, we obtain that
\begin{align}\label{est1}
\frac{1}{c^*}\|\Qr_\lambda f^\dag\|_{q}
\leq & \|(\sum_{j=0}^\infty|P_j \Qr_\lambda f^\dag|^2)^{\f 1 2}\|_{q}\\
=& \|(|P_{ \lfloor \lambda \rfloor+1}P_{ \lfloor \lambda \rfloor} f^\dag|^2+|(P_{ \lfloor \lambda \rfloor+1}+P_{ \lfloor \lambda \rfloor+2}) P_{ \lfloor \lambda \rfloor+1} f^\dag|^2+\sum_{j=  { \lfloor \lambda \rfloor}+2}^\infty |P_j f^\dag|^2)^{\f 1 2}\|_{q}\notag.
\end{align}
From Proposition \ref{pro:R}, it follows that the finite set $\{ P_{ \lfloor \lambda \rfloor+1},P_{ \lfloor \lambda \rfloor+1}+P_{ \lfloor \lambda \rfloor+2},\textup{id}\}$ is $\Rs$-bounded with
\begin{align*}
\Rs(\{ P_{ \lfloor \lambda \rfloor+1},P_{ \lfloor \lambda \rfloor+1}+P_{ \lfloor \lambda \rfloor+2},I\})\leq&
\Rs(\{P_{ \lfloor \lambda \rfloor+1}\})+\Rs(\{P_{ \lfloor \lambda \rfloor+1}+P_{ \lfloor \lambda \rfloor+2}\})+\Rs(\{\textup{id}\})\\
\leq& C_R:=1+3C_G\sup_{j\geq 0}\|P_j\|_{B(L^q(\Omega,\mu;\C))}.
\end{align*}
Let now  $N\in \mathbb{N}$ be arbitrarily fixed with $N> {\lfloor \lambda \rfloor}$. According to the definition of the $\mathcal R$-boundedness (see Definition \ref{def:R}),
by choosing
  \begin{align*}
n:=N-\lfloor \lambda \rfloor+1, \quad T_1:=P_{\lfloor \lambda \rfloor+1}, \quad
T_2:=P_{\lfloor \lambda \rfloor+1}+P_{\lfloor \lambda \rfloor+2},\quad
T_k:=  \textup{id} \quad \forall k=3,\ldots, n,
\end{align*}
and $z_k:=P_{\lfloor \lambda \rfloor +k-1} f^\dag$ for all $k=1,\ldots n$  in \eqref{def:R0}, we obtain
\begin{align*}
&\|(|P_{\lfloor \lambda \rfloor+1}P_{\lfloor \lambda \rfloor} f^\dag|^2+|(P_{ \lfloor \lambda \rfloor+1}+P_{\lfloor \lambda \rfloor+2}) P_{\lfloor \lambda \rfloor+1} f^\dag|^2+ \!\!\!\! \sum_{j=  {\lfloor \lambda \rfloor}+2}^N \!\!\!\!\! |P_j f^\dag|^2)^{\f 1 2}\|_{q}
\leq C_R \|(\sum_{j= \lfloor \lambda \rfloor }^N |P_j f^\dag|^2)^{\f 1 2}\|_{q}.
\end{align*}
Since $N$ was chosen   arbitrarily,   it follows that
\begin{align}\label{eq:rbound0}
&\|(|P_{\lfloor \lambda \rfloor+1}P_{\lfloor \lambda \rfloor} f^\dag|^2+|(P_{ \lfloor \lambda \rfloor+1}+P_{\lfloor \lambda \rfloor+2}) P_{\lfloor \lambda \rfloor+1} f^\dag|^2+\sum_{j={\lfloor \lambda \rfloor}+2}^\infty  |P_j f^\dag|^2)^{\f 1 2}\|_{q}   \\
\leq& C_R \|(\sum_{j=\lfloor \lambda \rfloor }^\infty |P_j f^\dag|^2)^{\f 1 2}\|_{q}
\leq \f{C_R}{2^{(\lambda-1) s \theta  }}\|(\sum_{j=0}^\infty 2^{2js\theta }|P_j f^\dag|^2)^{\f 1 2}\|_{q}
=   \f{2^{s\theta }C_R}{2^{\lambda s\theta}}\|f^\dag\|_{F^{s\theta}_q},\notag
\end{align}
where we have used the definition \eqref{Fspace} for the last identity.  Combinig \eqref{est0} and \eqref{est1}-\eqref{eq:rbound0}   results in
\beq\label{est2}
{\bf I_1}
\leq \f{c_p}{2\hat p} \| \xdag-x\|_{p}^{\hat{p}}+  \f{C}{2^{\lambda \hat q s\theta }}\|f^\dag\|_{F^{s\theta}_{q}}^{\hat{q}} \quad \forall \, x\in D(T),
\eeq
for some $C>0$, depending only on $c^*$, $c_p$, $\hat{q}$, $s$, $\theta$ and $C_R$.

Next, we     estimate   the second term   ${\bf I_2}$ by applying \eqref{eq:pair} and \eqref{eq:cond} to \eqref{decomp}:
\beq\label{est3}
{\bf I_2} = \la \xdag-x, \Pr_\lambda f^\dag \ra_{p,q} \lesssim  \|\Pr_\lambda f^\dag\|_{  F^{\theta}_{q}}  \Psi_0(\|T(\xdag)-T(x)\|_Y).
\eeq
Let us now derive an appropriate upper bound for  $\|\Pr_\lambda f^\dag\|_{ F^\theta_{q}}$.
Similar to \eqref{proj0}, invoking    the almost orthogonality \eqref{almost} and  the   partition of identity \eqref{partition}, we deduce   that
\beq\label{proj1}
 P_j\Pr_{\lambda} z= P_j\sum_{k=0}^{\lfloor \lambda \rfloor} P_k z=
 \begin{cases}
  0,\quad &j\geq \lfloor \lambda \rfloor+2,\\
  P_{\lfloor \lambda \rfloor+1} P_{\lfloor \lambda \rfloor}  z, \quad &j={\lfloor \lambda \rfloor}+1,\\
P_{\lfloor \lambda \rfloor}   (P_{\lfloor \lambda \rfloor}+P_{\lfloor \lambda \rfloor-1}) z\quad &j={\lfloor \lambda \rfloor},\\
  P_jz,  \quad  &j\leq \lfloor \lambda \rfloor -1  \\
 \end{cases} \quad
\eeq
holds true for all $z\in L^q(\Omega,\mu)$.
Since the finite set
$\{  P_{\lfloor \lambda \rfloor},P_{\lfloor \lambda \rfloor}+P_{\lfloor \lambda \rfloor-1},\textup{id}\}$ is $\Rs$-bounded with $\Rs(\{  P_{\lfloor \lambda \rfloor},P_{\lfloor \lambda \rfloor}+P_{\lfloor \lambda \rfloor-1},\textup{id}\})\leq 1+3C_G\sup_{j\geq 0}\|P_j\|_{B(L^q(\Omega,\mu;\C))}=C_R$,  using \eqref{proj1} and analogous arguments for \eqref{eq:rbound0}, we infer   that
\begin{align}\label{eq:p}
&\|\Pr_\lambda f^\dag\|_{F^{\theta}_q}
=\|(\sum_{j=0}^\infty 2^{2j\theta }|P_j \Pr_\lambda  f^\dag |^2)^{\f 1 2}\|_{q}\\
=&\|(\sum_{j=0}^{\lfloor \lambda\rfloor-1} 2^{2j\theta}|P_j   f^\dag|^2+ 2^{2\lfloor \lambda \rfloor\theta }|(P_{\lfloor \lambda \rfloor}+P_{\lfloor \lambda \rfloor-1})P_{\lfloor \lambda \rfloor}f^\dag|^2+  2^{2(\lfloor \lambda \rfloor+1)\theta }|P_{\lfloor \lambda \rfloor} P_{\lfloor \lambda \rfloor+1} f^\dag|^2
)^{\f 1 2}\|_{q} \notag\\\notag
\leq & C_R \|(\sum_{j=0}^{\lfloor \lambda \rfloor+1} 2^{2j\theta}|P_j   f^\dag|^2)^{\f 1 2}\|_{q}
\leq C_R 2^{(\lambda+1)\theta(1-s) } \|(\sum_{j=0 }^{\lfloor \lambda \rfloor+1} 2^{2s j\theta }|P_j   f^\dag|^2)^{\f 1 2}\|_{q}
\leq  C_R 2^{(\lambda+1)\theta(1-s) } \|f^\dag\|_{F^{s\theta}_{q}}.
\end{align}
Applying \eqref{eq:p}  to \eqref{est3} leads to
\beq\label{est4}
{\bf I_2} \lesssim  2^{(\lambda+1) \theta (1-s) } \|f^\dag\|_{F^{s\theta}_{q}}\Psi_0(\|T(\xdag)-T(x)\|_Y).
\eeq
Finally,   combining \eqref{decomp}, \eqref{est2}, and \eqref{est4} together, we arrive at
\begin{align*}
&\la \xdag-x,  f^\dag\ra_{p,q}\leq \f{c_p}{2\hat p} \| \xdag-x\|_{p}^{\hat{p}}+  \\
&C\inf_{\lambda\geq 1}\left(\f{1}{2^{\lambda \hat q s\theta }}\|f^\dag\|_{F^{s\theta}_{q}}^{\hat{q}}  +2^{(\lambda+1)\theta(1-s) }\|f^\dag\|_{F^{s\theta}_q}\Psi_0(\|T(\xdag)-T(x)\|_Y)\right) \quad \forall\,x\in D(T),
\end{align*}
for  all sufficiently large  $C>0$, independent of $x$. The function $\Psi:(0,\infty)\to (0,\infty)$ defined by
\beq\label{def:psi}
\Psi(\delta):=C\inf_{\lambda\geq 1}\left(\f{1}{2^{\lambda \hat q s\theta }}\|f^\dag\|_{F^{s\theta}_{q}}^{\hat{q}}  +2^{(\lambda+1)\theta(1-s) }\|f^\dag\|_{F^{s\theta}_q}\Psi_0(\delta)\right)
\eeq
is   concave, continuous and strictly increasing   (cf. the proof of \cite[Theorem 4.3]{chen2018variational}). In conclusion,  VSC \eqref{VSCp} holds true for $\beta = \frac{c_p}{2}$ and  the concave index function \eqref{def:psi} for all sufficiently large $C>0$.

Eventually,  {since $s, \hat{q}, \theta$ are fixed} and $\lim_{\delta \to 0}\Psi_0(\delta)=0$,  if $\delta$ is small enough,  there exists $\lambda_0\geq 1$ such that
${\f{1}{2^{\lambda_0\theta}}}=\Psi_0(\delta)^{\f{1}{1+(\hat q-1)s}}$, which implies that
$(\f{1}{2^{\lambda_0 \theta }})^{\hat q s}=
\Psi_0(\delta)^{\f{\hat qs}{1+(\hat q-1)s}}$ and
$ (\f{1}{2^{\lambda_0\theta}})^{s-1}\Psi_0(\delta)=\Psi_0(\delta)^{\f{s-1}{1+(\hat q-1)s}}\Psi_0(\delta)
=\Psi_0(\delta)^{\f{\hat qs}{1+(\hat q-1)s}}
$.
Therefore,  if $\delta$ is small enough, \eqref{def:psi} yields that
$$
\Psi(\delta) \lesssim \f{1}{2^{\lambda_0 \hat q s\theta }}\|f^\dag\|_{F^{s\theta}_{q}}^{\hat{q}}  +2^{(\lambda_0+1)\theta(1-s) }\|f^\dag\|_{F^{s\theta}_q}\Psi_0(\delta) =    (\|f^\dag\|_{F^s_q}^{\hat q}+2^{\theta(1-s) }\|f^\dag\|_{F^s_{q}})\Psi_0(\delta)^{\f{\hat qs}{1+(\hat q-1)s}}.
$$
This completes the proof.
\end{proof}
Let us close this section by presenting an exemplary  application of
Theorem \ref{the:vsc} with an   optimal convergence rate. We consider $p=q=2$ and
an unbounded, self-adjoint, and strictly  positive operator $A:D(A)\subset L^2(\Omega,\mu;\C)\to L^2(\Omega,\mu;\C)$.  By      the functional calculus for self-adjoint operator (see, e.g., \cite[Lemma 6.1(2)]{kriegler2016paley}), $A:D(A)\subset L^2(\Omega,\mu; \C)\to L^2(\Omega,\mu;\C)$ is a 0-sectorial operator with   $\mathcal{M}^\eta$-calculus for some $\eta>0$.  Therefore,  in view of  Lemma \ref{lemma:m}, $L^2(\Omega,\mu; \mathbb C)$ admits an LP-decomposition $\mathcal{P}=\{P_j\}_{j=0}^\infty$  such that
\beq\label{equvi:F}
F^\theta_2(\mathcal{P})=D(A^\theta)\quad \forall \, \theta\geq  0,
\eeq
and its dual space $F^\theta_2(\mathcal{P})^*$ coincides with  $D(A^{-\theta})$.  We suppose that the forward operator $T: D(T)\subset L^2(\Omega,\mu)\to L^2(\Omega,\mu)$  is   linear, and there exists  $\theta\ge0$ such that
$$
\|x\|_{D(A^{-\theta})}\lesssim \|T x\|_{ L^2(\Omega,\mu)} \quad \forall x\in D(T).
$$
Now, if $x^*=0$ and $\xdag\in D(A^{s\theta})$ for some $0<s\leq 1$,
then Theorem \ref{the:vsc}   yields that  VSC \eqref{VSCp} holds true for the index function
$$
\Psi(\delta)  :=     C \delta^{\f{2s}{1+ s}}
$$
for any sufficiently  large $   C>0$.  Eventually, Corollary \ref{corconv}  yields the following convergence rate
\begin{equation} \label{optimrate}
\|x^\delta_{\alpha(\delta)}-x^\dag\|_{L^2(\Omega,\mu)}=O(\delta^{\f{s}{1+s}})  \quad \text{as}\,\, \delta\to 0^+
\end{equation}
for the Tikhonov regularization method \eqref{Tik} with $\ell=2$ and the parameter choice $\alpha(\delta):=   C^{-1} \delta^{\f{2}{s+1}}
$.  It is well-known that the convergence rate \eqref{optimrate}   is optimal (see   \cite{Natterer} or   \cite[Theorem 1.1.]{Tautenhahn}).

\section{Parameter identification of elliptic equations with measure data in the $L^p$-setting}\label{sec:app}

Throughout this section, let $\Omega\subset \R^n$ ($n \ge 2$) be a bounded $C^{1,1}$ domain, and let $\kappa\in L^\infty(\Omega)$ be a real-valued function satisfying
\beq\label{eq:kappa}
0<\lambda_0\leq \kappa(x)\leq \Lambda \,\,\text{for a.e. }  x \in \Omega,
\eeq
with two positive real constants $\lambda_0<\Lambda$.  We consider the inverse problem of reconstructing  the possibly unbounded diffusion coefficient
$a:\Omega\to \R$  of  the following elliptic equation:
\beq\label{elliptic}
\begin{cases}
\nabla(\kappa\nabla u)+au=\mu_\Omega\quad &\text{in}\, \Omega, \\
\kappa\frac{\p u}{\p \nu}=\mu_\Gamma\quad &\text{on}\, \Gamma:=\partial \Omega,
\end{cases}
\eeq
where $\mu_\Omega$ and $\mu_\Gamma$  are     regular signed Borel  measures on
$\Omega$ and
$\Gamma$.

\begin{Definition}\label{def:ws}
Let $ \mu_\Omega+\mu_{\Gamma}=: \mu_{\overline \Omega} \in C(\overline \Omega)^*$ be a  regular signed Borel  measure on $\overline \Omega$. A function $u\in H_1^1(\Omega)$ is said to be a  weak solution of
\eqref{elliptic}  if $au\in L^1(\Omega)$ and
\beq
\int_\Omega \kappa \nabla u\cdot \nabla \varphi
+ au\varphi dx =\int_{\overline{\Omega}} \varphi d\mu_{\overline{\Omega}}\quad \forall \varphi\in C^\infty(\overline{\Omega}).
\eeq
\end{Definition}
 The well-posedness of \eqref{elliptic} requires the following {\rm ellipticity condition}:

\begin{enumerate}

\item[{\bf (EC)}$_m$]

 Let $p>n/2$ and suppose that  $a \in L^p (\Omega)$  is a   nonnegative function satisfying  \beq\label{ec}
 \int_\Omega(\kappa |\nabla \varphi|^2+ a|\varphi|^2)dx \geq
 m\|\varphi\|_{H^1_2(\Omega)}^2\quad \forall\,\varphi\in H^1_2(\Omega),
 \eeq
for some $m>0$.
\end{enumerate}
We note that Proposition  \ref{pro:sobolev} ((i) and (ii))  implies that  $H^1_2(\Omega)\embed L^{\f{2n}{n-2}}(\Omega)$, if $n\geq 3$, and $H^1_2(\Omega)\embed L^{s}(\Omega)$  for all $1<s<\infty$, if $n=2$. Thus, the requirement   $p>\f{n}{2}$  is reasonable  to ensure that the second term in the left hand side of \eqref{ec} is well-defined.

\begin{Theorem}[{\cite[Theorem 4]{alibert1997boundary}}]\label{the:well}
Let $p>\f{n}{2}$ and $a\in L^p(\Omega)$ satisfying {\bf (EC)}$_m$ for some $m>0$. Then, for every  $\mu_{\overline{\Omega}} \in C(\overline \Omega)^*$,  the elliptic problem \eqref{elliptic}  admits a unique weak solution $u \in H^1_\tau(\Omega)$  for all $1\leq \tau<\f{n}{n-1}$. Moreover, for every $1\leq \tau<\f{n}{n-1}$, there exists a constant $C(m,\tau)>0$, independent of  $a$, $\mu_{\overline{\Omega}}$, and $u$, such that
\beq\label{eq:est}
\|u\|_{H^1_\tau(\Omega)}\leq C(m,\tau)\|\mu_{\overline{\Omega}}\|_{M(\overline{\Omega})}.
\eeq
\end{Theorem}

\subsection{Existence and convergence} \label{sec: math pro}
In all what follows,  let   $\mu_{\overline{\Omega}} \in   C(\overline \Omega)^*$, $p>\f{n}{2}$, $\mathfrak{p}\geq p$, $\underline{a}>0$, $M>0$ be fixed  and
\beq\label{def:ds}
D(S):=\{a\in L^\pf(\Omega)\mid \|a\|_{L^\pf(\Omega)}\leq M   \  \textrm{and} \    \underline{a}  \le a (x)\,\text{ for a.e. } x \in \Omega\}.
  \eeq
 \begin{Lemma}\label{lemma:conv}
Suppose that $\{a_k\}_{k=1}^\infty \subset D(S)$, and,
for every $k \in \mathbb N$, let $u_k \in   H^1_\tau(\Omega)$ for all $1\leq \tau<\f{n}{n-1}$ denote the   unique weak solution to \eqref{elliptic} associated with $a_k$. Then,
$$
a_k \wto a\ \textrm{ weakly in} \ L^p(\Omega) \quad \Rightarrow   \quad u_k  \wto u  \textrm{ weakly in } H^1_\tau(\Omega) \textrm{ for all } 1 \le \tau< \frac{n}{n-1},
$$
where $u \in H^1_\tau(\Omega)$    is the unique weak solution to \eqref{elliptic} associated with $a \in D(S)$.
\end{Lemma}
\begin{proof}
Suppose that the sequence $\{a_k\}_{k=1}^\infty \subset D(S)$ converges weakly in $L^p(\Omega)$ towards some element $a \in L^p(\Omega)$.
Since $D(S)$  is a weakly compact set in  $L^{\pf}(\Omega)$ and the embedding $L^{\pf}(\Omega)\embed L^p(\Omega)$ is continuous,  it follows that the set $D(S)$ is a weakly compact set  in $L^p(\Omega)$, which yields that $a\in D(S)$.   Furthermore, Theorem \ref{the:well} ensures that for every fixed  $1< \tau<\f{n}{n-1}$, there exists a subsequence $\{u_{k_m}\}_{m=1}^\infty \subset \{u_k\}_{k=1}^\infty$ weakly converging  in $H^1_{\tau}(\Omega)$  to some $u\in  H_\tau^1(\Omega)$.

Let us now fix a $\tau \in (\f{np}{n(p-1)+p}, \f{n}{n-1})$, which ensures that  $\f{n\tau }{n-\tau }>\f{p}{p-1}$. Then,
  Proposition \ref{pro:sobolev} (i)  implies that    the embedding
  $H^1_{\tau}(\Omega)\embed L^\f{p}{p-1}(\Omega)$ is compact.  For this reason, we obtain the strong convergence $u_{k_m}\to u$ in $L^\f{p}{p-1}(\Omega)$, which yields the weak convergence
$a_{k_m}u_{k_m}\wto a u$ in $L^1(\Omega)$. Thus,
for any $\varphi\in C^\infty(\overline{\Omega})$, we obtain that
$$
\int_\Omega \kappa\nabla u \cdot \nabla \varphi+a u\varphi dx=\lim_{m\to\infty}
\int_\Omega \kappa\nabla  u_{k_m}\cdot \nabla \varphi+a_{k_m}u_{k_m}  \varphi dx=\int_{\overline{\Omega}} \varphi d\mu_{\overline{\Omega}}.
$$
It follows therefore from Theorem \ref{the:well}   that $u$ is the unique weak solution to \eqref{elliptic}, and   so a well-known argument
implies that the whole sequence $\{u_k\}_{k=1}^\infty$  converges weakly  in $H^1_{\tau}(\Omega)$   towards  $u$.  Finally, as the embedding $H^1_\tau(\Omega) \embed H^1_{\tilde \tau}(\Omega) $ for any $\tilde \tau \in [1,\tau]$ is linear and bounded, we conclude that $\{u_k\}_{k=1}^\infty$  converges weakly  in $H^1_{\tau}(\Omega)$ for  all $1\le \tau < \frac{n}{n-1}$ towards  $u$.
\end{proof}

In view of Theorem \ref{the:well}, we    introduce  the solution operator $S: D(S) \subset L^p(\Omega)\to  Y$, $a\mapsto u$,  where $Y$ denotes a real Banach space satisfying    $   H^1_\tau(\Omega) \embed Y$  for some $1\leq \tau<\f{n}{n-1}$. More precisely,    the operator $S$  assigns to every coefficient $a\in D(S)$  the unique weak solution    of \eqref{elliptic}  $u \in H^1_\tau(\Omega)$  for all $1\leq \tau<\f{n}{n-1}$.
   Applying the solution operator,   the mathematical formulation of the elliptic inverse coefficient problem  \eqref{elliptic} reads as follows: Find $a \in D(S)$ such that
 \beq\label{eq:ip}
S(a)=u^\dag,
\eeq
where  $u^\dag \in H^1_\tau(\Omega)$ for all $1 \le \tau < \frac{n}{n-1}$ denotes the  unique weak solution of \eqref{elliptic} associated with  the  true coefficient $a^\dag\in D(S)$.  For our convergence analysis,  we assume that the (unknown) true solution
$a^\dag$  is the $L^p$-norm minimizing solution in the sense that $a^\dag\in D(S)$ solves
\beq\label{min:lp}
\|a^\dag-a^*\|_{L^p(\Omega)}=\min_{a\in \Pi({u^\dag})} \|a-a^*\|_{L^p(\Omega)} \quad \textrm{with }  \Pi(u^\dag):=\{a\in  D(S) \mid S(a)=u^\dag\}.
\eeq
\begin{Lemma}
The  nonempty set
$
\Pi(u^\dag)
$	
is  bounded, convex, and closed in $L^p(\Omega)$. Therefore,  the minimization problem \eqref{min:lp} admits a unique solution.
\end{Lemma}

\begin{proof}  The boundedness follows immediately from the definition of $D(S)$ (see \eqref{def:ds}) and $L^{\pf}(\Omega)\embed L^p(\Omega)$. Moreover,
by Definition \ref{def:ws},  it is straightforward to show that $\Pi(u^\dag)$ is   convex.    Let  us now prove that $\Pi(u^\dag) \subset L^p(\Omega)$ is closed. To this aim, let
$ \{a_k\}_{k=1}^{\infty}\subset\Pi(u^\dag)$ such that $  a_k\to a$ in $L^p(\Omega)$.   This weak limit satisfies  $a\in D(S)$ since $D(S) \subset L^p(\Omega)$ is   weakly compact  (cf. the proof of Lemma \ref{lemma:conv}). Furthermore,  as    the embedding
  $H^1_{\tau}(\Omega)\embed L^\f{p}{p-1}(\Omega)$ holds true for all $\f{np}{n(p-1)+p}<\tau <\f{n}{n-1}$ (cf. the proof of Lemma \ref{lemma:conv})  we obtain that  $u^\dag\in L^{\f{p}{p-1}}(\Omega)$, which implies that
$$
a_ku^\dag \to a u^\dag\quad \textup{in}\, L^1(\Omega),
$$
and consequently
$$
\int_{\Omega}\kappa\nabla u^\dag\cdot \nabla \varphi+ a u^\dag\varphi dx=\lim_{k\to\infty}
\int_{\Omega}\kappa\nabla u^\dag\cdot \nabla \varphi+ a_k  u^\dag\varphi dx
=\int_\Omega \varphi d\mu_{\overline{\Omega}}\quad\forall \,\varphi\in C^\infty(\overline{\Omega}).
$$
In conclusion, $a \in \Pi(u^\dag)$. This completes the proof.
\end{proof}

Now, given $\alpha>0$, the Tikhonov regularization problem associated with \eqref{eq:ip}  reads as
 \beq\label{mina}
\min_{a\in D(S)}\left(\f{1}{\ell}\| S(a)-u^\delta\|_Y^\ell+\f{\alpha}{\hat p}\|a-a^*\|_{{L^p(\Omega)}}^{\hat{p}}\right),
\eeq
for a fixed constant $\ell>1$, $\hat{p}=\max\{p,2\}$ and  an arbitrarily fixed a priori estimate $a^*\in L^p(\Omega)$ for  $a^\dag$.  Moreover, the noisy data
$u^\delta$ satisfy
$$
\|u^\dag-u^\delta\|_Y\leq \delta,
$$
with the noisy level $\delta>0$.   From
the classical theory of Tikhonov regularization (see e.g. \cite{hofmann2007convergence,schuster2012regularization}), the sequentially weak-to-weak continuity result (Lemma \ref{lemma:conv})   implies    the following existence and plain convergence results:

\begin{Theorem}  \label{tikop}The following assertions hold true:
\begin{enumerate}
\item[\textup{(i)}]For each $\alpha>0$ and $u^\delta\in Y$,   \eqref{mina} admits a solution $a^\delta_\alpha \in D(S)$.

\item[\textup{(ii)}] Let $\{\delta_k\}_{k=1}^\infty \subset (0,+\infty)$ be a  null sequence  and $\{u^{\delta_k}\}_{k=1}^\infty\subset Y$ be a sequence  satisfying
$$
\|u^{\delta_k}-u^\dag\|_Y\leq \delta_k\quad \forall\,k\in\mathbb{N}.
$$
Moreover, let $\{\alpha_k\}_{k=1}^\infty\subset (0,+\infty)$ fulfill
$$
\alpha_k\to 0,\,\,\f{\delta_k^\ell}{\alpha_k}\to 0,
$$
where $\ell\geq 1$ is as in \eqref{mina}.
If $a_k$ is a minimizer of \eqref{mina} with $u^\delta$ and $\alpha$ replaced by $u^{\delta_k}$ and $\alpha_k$,  respectively,  then $a_k$ converges strongly to  $a^\dag$ in $L^p(\Omega)$.
\end{enumerate}
\end{Theorem}

\subsection{VSC for  \eqref{mina}}\label{sec:main}

Our goal is to verify VSC   for  the   Tikhonov regularization problem  \eqref{mina}. We  shall apply our abstract result (Theorem \ref{the:vsc}) to the case of $T=S$ and show that  the conditional estimate  \eqref{eq:cond} is satisfied.

\begin{Theorem}\label{first}
Let $\pf>\f{n}{2}$,    \beq\label{eq:tau}
\tau\in\begin{cases}\, (1,+\infty)\, &\text{if}\ \pf\geq n,\\
(\f{\pf n}{n\pf-n+\pf},\f{\pf n}{n-\pf })\, &\text{if}\ \f{n}{2}<\pf<n,
\end{cases} \eeq
and $1<r,q<+\infty $, $\gamma>0$ such that
\begin{enumerate}

\item[\textup{(a)}]   $u^\dag \in H^1_r(\Omega)$ and
$|u^\dag |\geq \gamma$ a.e. in $\Omega$;

\item[\textup{(b)}]   $S(a)-S(a^\dag)\in {H^1_\tau(\Omega)} $ for all $ a\in D(S)$;

\item[\textup{(c)}] \beq\label{ineq:qr}
1-\f{1}{\tau}=\f{1}{q}+\f{1}{r}.
\eeq
\end{enumerate}
Furthermore,     $p:=\f{q}{q-1}$,  $\hat{p}:=\max\{2,p\}$, $\hat{q}:=\min\{2,q\}$, and suppose that $J_{\hat p}(a^\dag-a^*) := f^\dag \in H^{s}_q(\Omega)$ for some
$s\in (0,1]$.  Then the following assertions hold true:

\begin{enumerate}

\item[\textup{(i)}]
There exists a  constant $C>0$ such that
\beq\label{key:est}
\|a-a^\dag\|_{H_q^1(\Omega)^*}\leq C\|S(a)-S(a^\dag)\|_{H^1_\tau(\Omega)}\quad \forall\, a\in D(S).
\eeq

 \item[\textup{(ii)}]

  If $\tau<\f{n}{n-1}$ and $Y=H_\tau^1(\Omega)$,  then VSC \eqref{VSCp} holds true for $T=S$,
    $\beta = \frac{c_p}{2}$, and $\Psi$ as in \eqref{def:psitheo} with $\theta=1$ and $\Psi_0(\delta)=\delta$.


 \item[\textup{(iii)}] If, in addition, there exist   $\overline\tau>\tau$ and $M_1>0$ such that
\beq\label{estover}
\|S(a)-S(a^\dag)\|_{H_{\overline\tau }^1(\Omega)}\leq M_1 \quad \,\,\forall \, a\in D(S),
\eeq
  and $Y=H_1^1(\Omega)$, then VSC  \eqref{VSCp} holds true
 for $T=S$, $\beta = \frac{c_p}{2}$, and $\Psi$ as in \eqref{def:psitheo} with $\theta=1$ and $\Psi_0(\delta)=\delta^\f{\overline\tau-\tau}{\tau (\overline\tau-1)}$.

 \item[\textup{(iv)}] If there exists $M_2>0$ such that
\beq\label{estover2}
\|S(a)-S(a^\dag)\|_{H_{\tau }^2(\Omega)}\leq M_2 \quad \,\,\forall \, a\in D(S),
\eeq
 and $Y=L^\tau (\Omega)$, then   VSC  \eqref{VSCp} holds true
 for $T=S$, $\beta = \frac{c_p}{2}$, and $\Psi$ as in \eqref{def:psitheo} with $\theta=1$ and $\Psi_0(\delta)={\delta}^{\frac{1}{2}}$.

\end{enumerate}

\end{Theorem}

\begin{Remark}\label{concremark}    \normalfont
~\begin{enumerate}
\item[\textup{(i)}]
The condition  {(a)} implies that $\Pi(u^\dag)=\{a^\dag\}$, and so the inverse problem \eqref{eq:ip} has a unique solution.  We note that the generalized duality map $J_{\hat p} : L^p(\Omega) \to L^q(\Omega)$ satisfies  $J_{\hat{p}}(w)(x)=\|w\|_p^{\hat p -p}|w(x)|^{p-2} w(x)$ if $w(x) \neq 0$ and $J_{\hat{p}}(w)(x)=0$ if $w(x)=0$   (see, e.g., \textup{ \cite[Section 1.1]{Barbu}}).  Therefore, the regularity assumption   $f^\dag:= J_{\hat p}(a^\dag-a^*)  \in H^{s}_q(\Omega)$ can be immediately translated to
  the difference between the true solution and the intial guess as follows:
 \beq\label{sp}
 \chi_{\omega} |a^\dag-a^*|^{p-2}(a^\dag-a^*)   \in H^{s}_q(\Omega)
  \eeq
where  $\chi_{\omega}$ denotes the characteristic function of $\omega:= \{x \in \Omega\, |\, a^\dag(x) \neq a^*(x) \}$.  {  Introducing
$$
\varphi_p(t):=\begin{cases}
|t|^{p-2} t \quad t\in \mathbb{R}\backslash \{0\}\\
0 \quad t=0
\end{cases}
\quad \textrm{and} \quad T_{\varphi_p} f:=\varphi_p\circ f,
$$
\eqref{sp} can be reformulated as
\beq\label{sp1}
T_{\varphi_p}(a^\dag-a^*)\in H^{s}_q(\Omega) \quad \iff \quad  a^\dag-a^*\in  T_{\varphi_p}^{-1}(H^{s}_q(\Omega))=T_{\varphi_p^{-1}}(H^{s}_q(\Omega)).
\eeq
In the literature, there exist numerous contributions to the superposition in  Lizorkin--Triebel spaces (see, e.g., \cite{BMW2008,BMW2010,Runst,Sickel}).
To characterize $T_{\varphi_p^{-1}}(H^{s}_q(\Omega))$, we first take advantage of  \cite[Theorem 3]{Sickel}. By an elementary calculation  (see Appendix), it holds that  
\beq\label{varphplus}
\varphi_{p}(\R)\subset \R,\quad |\varphi^{(l)}_{p}(t)|\lesssim |t|^{p-1-l},\,\, l=0,\cdots,N, \,\, 
\eeq
and 
\beq\label{varphplus2}
\sup_{t_0\neq t_1}\frac{|\varphi^{(N)}_{p}(t_1)-\varphi^{(N)}_{p}(t_0) |}{|t_1-t_0|^{\tau}}
<\infty,
\eeq
 where $N\in \mathbb{Z}$ and $0<\tau:=p-1-N\leq 1$.  Thus, Theorem 3  in \cite{Sickel} is applicable, and \cite[Theorem 3, (i)]{Sickel} yields that 
 $$
 H^s_q(\R^n)\subset T_{\varphi_p}^{-1}(H^{s}_q(\R^n)) \quad \text{if}\quad p>\max\{2,n\} \,\,\textup{and}\quad
 n/q<s<p-1. 
 $$  
 Then, from the definition of $H^{s}_q(\Omega)$ and the fact that $T_{\varphi_p}^{-1}(U)\mid_{\Omega}
=T_{\varphi_p}^{-1}(u)$ whenever $U\in H_q^s(\R^n)$ and $u=U\mid_{\Omega}$, we   conclude that 
$$
 H^s_q(\Omega)\subset T_{\varphi_p}^{-1}(H^{s}_q(\Omega)) \quad \text{if}\quad p>\max\{2,n\} \,\,\textup{and}\quad
 n/q<s<p-1. 
 $$  By \cite[Theorem 3, (ii)]{Sickel} along with the argument above, we obtain that 
 $$
   H_q^{s^*}(\Omega)\subset T_{\varphi_p}^{-1}(H_q^{p-1}(\Omega)) 
\subset T_{\varphi_p}^{-1}(H_q^{s}(\Omega))
   $$
   if $p>2$, $0<s<p-1$, and   $n(1-2/p)<s^*\leq \min\{1+n(1-2/p),n(1-1/p)\}$,  where we have used the inclusion 
   $H^{p-1}_q(\Omega)\subset H^s_q(\Omega)$ (\cite[Page 42]{Yagi}).   Let us finally consider the case   $1<p<2$, $0<s<p-1$, and $t>\frac{s}{p-1}$. In this case,  introducing  $\varphi_{p,+}(t):=|t|^{p-1}$ for $t\in \R$,   \cite[Section 5.4.4, Theorem 1]{Runst} and the embedding result   \cite[Section 2.2.1, Proposition 1]{Runst} yield
\beq\label{varphplus3}
T_{\varphi_{p,+}}(H^t_{p}(\Omega))\subset H^s_q(\Omega).  
\eeq
For the convenience  of the reader, we present a proof for \eqref{varphplus3} in   Appendix. For the above case,   if  $a^\dag-a^*\geq 0$ a.e. on $\Omega$, and $a^\dag-a^*\in H^t_{p}(\Omega)$, then  $T_{\varphi_p}(a^\dag-a^*)=T_{\varphi_{p,+}}(a^\dag-a^*)\in H^s_q(\Omega)$.}

\item[\textup{(ii)}]
 Theorem \ref{the:well} implies that  $S(a),S(a^\dag)\in H_\tau^1(\Omega)$ for all $1<\tau<\f{n}{n-1}$. Nevertheless,  we shall show in   Lemmas \ref{lemma:prior} and \ref{lemma:prior2}  that $S(a)-S(a^\dag)$   enjoys a higher regularity property,   depending  on the regularity of $\kappa$, such that the assumptions  \textnormal{(b)}, \eqref{estover} and \eqref{estover2} are reasonable.

 \item[\textup{(iii)}]   The second part of  the assumption  \textnormal{(a)}  can be verified  under additional  assumptions: Suppose that  $a^\dag\in L^\infty(\Omega)$,  $\mu_\Omega\in L^{h_1}(\Omega)$ with  $h_1>\frac{n}{2}$, and $\mu_\Gamma\in L^{h_2}(\Gamma) $    with $h_2>n-1$. This additional regularity assumption implies that $u^\dag \in H^1_2(\Omega)\cap C(\overline{\Omega})$ (see \cite[Theorem 2]{alibert1997boundary}).
 Suppose further that $\Omega$ is convex and  piecewise smooth,  $\mu_\Omega \ge 0$ a.e. in $\Omega$,  and
 $\mu_\Gamma \geq \hat c$ a.e. on $\p\Omega$ for some constant $\hat c >0$. Let    $\hat u\in H^1_2(\Omega)\cap C(\overline{\Omega})$ denote the weak solution to the Neumann  problem
 $$
 -\nabla\cdot (\kappa \nabla \hat u)+  a^\dagger \hat u=0\ \text{in}\ \Omega\quad \text{and} \quad \kappa \f{\p \hat u}{\p \nu}=\hat c\ \text{on}\ \Gamma.
 $$
We note again that the continuity of $\hat u$ follows from   \cite[Theorem 2]{alibert1997boundary}. Now, according to \textup{\cite[Theorems 2 and 4]{Le}}, there exists a constant $\gamma>0$ such that $\hat u(x) \ge \gamma$ for all $x \in \overline \Omega$.   Letting $w:=u^\dag- \hat u$, we obtain   that
 \begin{align*}
 \int_{\Omega} \kappa \nabla w\cdot \nabla \varphi+ a^\dag w\varphi dx-\int_{\Gamma} \kappa (\mu_\Gamma -\hat c) \varphi ds=\int_\Omega  \mu_\Omega \varphi dx  \quad \forall \varphi\in H^{1}_2(\Omega),
 \end{align*}
from which it follows that
  \begin{align} \label{iqfinal}
 \int_{\Omega} \kappa \nabla w\cdot \nabla \varphi+ a^\dag w\varphi dx=\int_\Omega \mu_\Omega \varphi dx
 +\int_{\Gamma}   \kappa (\mu_\Gamma -\hat c)  \varphi ds
 \geq 0
 \end{align}
  for all non-negative functions $\varphi\in H^{1}_2(\Omega)$.   By   \textup{\cite[Theorem 2]{Le}}, it follows  from \eqref{iqfinal}  that $w(x) \ge 0$ holds for all $x \in \overline \Omega$. As a consequence, we obtain that
 $u^\dag(x) \ge   \hat u(x) \ge \gamma$   for all  $x \in \overline \Omega$.
\end{enumerate}
\end{Remark}
\begin{proof}
(i)
Let $\tau^*$ denote the conjugate exponent of $\tau$, i.e., $\tau^*=\f{\tau}{\tau-1}$.
 For each $a\in D(S)$, by the definition of the weak solution, we have
$$
\int_{\Omega} \kappa\nabla(S(a)-S(a^\dag))\cdot \nabla \varphi+ a(S(a)-S(a^\dag))\varphi dx\notag
=\int_{\Omega} (a^\dag-a)S(a^\dag)\varphi dx \quad\forall \varphi\in C^\infty(\overline{\Omega}).
$$	
Then, in view of  \eqref{eq:kappa} and H\"{o}lder's inequality, it follows that
\begin{align}\label{inner0}
\left|\int_{\Omega} (a^\dag-a)S(a^\dag)\varphi dx\right|
\leq& \Lambda\|\nabla(S(a)-S(a^\dag))\|_{\tau}\|\nabla \varphi\|_{\tau^*}
+\left|\int_{\Omega} a(S(a)-S(a^\dag))\varphi dx\right|\notag\\
=:& \Lambda\|\nabla(S(a)-S(a^\dag))\|_{\tau}\|\nabla \varphi\|_{\tau^*}+{\bf J}.
\end{align}
{
By the defiiniton of $D(S) \subset L^\pf(\Omega)$ (see \eqref{def:ds}),  we have
\beq\label{eq:J0}
 {\bf J} \le \|a\|_{\frak p}\|S(a)\varphi-S(a^\dag)\varphi\|_{\f{\pf}{\pf-1}} \le M \|S(a)\varphi-S(a^\dag)\varphi\|_{\f{\pf}{\pf-1}}    \  \forall\,(a,\varphi)\in D(S)\times C^\infty(\overline{\Omega}).
\eeq
Let us now prove the following estimate for  ${\bf J}$:
\beq\label{eq:J}
{\bf J} \lesssim  \|S(a)-S(a^\dag)\|_{H^1_\tau(\Omega)}\|\varphi\|_{H^1_{\tau^*}(\Omega)} \quad  \forall\,(a,\varphi)\in D(S)\times C^\infty(\overline{\Omega}).
\eeq
We first consider the case $1<\tau,\tau^*<n$, which is  only possible  for $n\geq 3$. For this case,  generalized H\"{o}lder's inequality and  Proposition \ref{pro:sobolev} (i)  yield that
\begin{align}\label{inner1}
\|S(a)\varphi-S(a^\dag)\varphi\|_{\f{\pf}{\pf-1}} \underbrace{\lesssim}_{\frak p > \f{n}{2}}\|(S(a)-S(a^\dag))\varphi\|_{\f{n}{n-2}}\leq \|S(a)-S(a^\dag)\|_{\f{n\tau}{n-\tau}}\|\varphi\|_{\f{n\tau^*}{n-\tau^*}}\notag\\
\lesssim \|S(a)-S(a^\dag)\|_{H^1_\tau(\Omega)}\|\varphi\|_{H^1_{\tau^*}(\Omega)} \quad  \forall\,(a,\varphi)\in D(S)\times C^\infty(\overline{\Omega}).
\end{align}
Applying \eqref{inner1} to  \eqref{eq:J0} yields  \eqref{eq:J}.    If $\tau\geq n$ and $\tau^*\geq n$,   both  $H^1_\tau(\Omega)$ and $H^1_{\tau^*}(\Omega)$ are   embedded to $L^{s}(\Omega)$ for all $1<s<+\infty$  (Proposition \ref{pro:sobolev} (ii)), and consequently  H\"{o}lder's inequality implies
\beq
\|S(a)\varphi-S(a^\dag)\varphi\|_{\f{\pf}{\pf-1}}\lesssim \|S(a)-S(a^\dag)\|_{H^1_\tau(\Omega)}\|\varphi\|_{H^1_{\tau^*}(\Omega)}\quad  \forall\,(a,\varphi)\in D(S)\times C^\infty(\overline{\Omega}),
\eeq
which yields  \eqref{eq:J}. Now suppose that $\tau\geq n$ and $\tau^*<n$.
From \eqref{eq:tau}, we obtain that
\beq\label{pn}
\f{1}{\pf}<\f{1}{\tau}+\f{1}{n} \quad \Rightarrow \quad 1-\f{1}{\pf}>1-\frac{1}{\tau}-\frac{1}{n} \quad \Rightarrow
\quad \f{\pf}{\pf-1}<\f{n\tau^*}{n-\tau^*}.
\eeq
Therefore,   in view of generalized H\"older's inequality and Proposition \ref{pro:sobolev}, we can choose $1<s<+\infty$ such that
\begin{align*}
\|S(a)\varphi-S(a^\dag)\varphi\|_{\f{\pf}{\pf-1}}\leq&  \|S(a)-S(a^\dag)\|_{s}\|\varphi\|_{\f{n\tau^*}{n-\tau^*}}
\lesssim \|S(a)-S(a^\dag)\|_{H^1_\tau(\Omega)}\|\varphi\|_{H^1_{\tau^*}(\Omega)},
\end{align*}
for all $(a,\varphi)\in D(S)\times C^\infty(\overline{\Omega})$. Thus, applying the above inequality to \eqref{eq:J0} gives \eqref{eq:J}.
Similarly,  \eqref{eq:J} is obtained for the  case of $\tau<n$ and $\tau^*\geq n$ as $
\f{\pf}{\pf-1}<\f{n\tau }{n-\tau}
$ is satisfied in this case.
}

Applying    \eqref{eq:J} to \eqref{inner0}, we obtain
\begin{align}\label{dualest}
\left|\int_{\Omega} (a^\dag-a)S(a^\dag)\varphi dx\right|\lesssim
\|S(a)-S(a^\dag)\|_{H^1_\tau(\Omega)}\|\varphi\|_{H^1_{\tau^*}(\Omega)}\, \forall\,(a,\varphi)\in D(S)\times H^1_{\tau^*}(\Omega),
\end{align}
where we have also used the density of $ C^\infty(\overline{\Omega})$ in  $H^1_{\tau^*}(\Omega)$ (cf. \cite{grisvard2011elliptic}).
On the other hand, we observe that
\begin{align}\label{dualob}
\|(a^\dag-a)\|_{H^1_{q}(\Omega)^*}=&\sup_{\|g\|_{H^1_q(\Omega)}=1}\left|\int_{\Omega} (a^\dag-a)g dx\right|
=\sup_{\|g\|_{H^1_q(\Omega)}=1}\left|\int_{\Omega} (a^\dag-a)S(a^\dag)\f{1}{S(a^\dag)} g dx\right|.
\end{align}
Now we show that   $\f{1}{S(a^\dag)}$ is well-defined in $H^1_{\tau^*}(\Omega)$. From  the condition (a), it follows that
$
\f{1}{S(a^\dag)}=F(S(a^\dag))
$
  holds true for a globally Lipschitz function $F:\mathbb{R}\to \mathbb{R}$ satisfying
$F(0)=0$ and $F(x)=\f{1}{x}$ for all $|x|\geq \gamma$. For this reason, Proposition \ref{pro:chain} (i) implies that $\f{1}{S(a^\dag)}\in H^{1}_{r}(\Omega)$. Furthermore, using Proposition \ref{pro:chain} (ii) and the condition (c), we  have
\beq\label{eq:g}
\left\|\f{1}{S(a^\dag)} g\right\|_{H^1_{\tau^*}(\Omega)}\leq C
\left\|\f{1}{S(a^\dag)}\right\|_{H^1_r(\Omega)}\|g\|_{H^1_q(\Omega)}\quad \forall g\in H_q^1(\Omega).
\eeq	
As a consequence of \eqref{dualest} and \eqref{dualob},
\beq\label{eq:amina}
\|(a^\dag-a)\|_{H^1_{q}(\Omega)^*}\lesssim  \sup_{\|g\|_{H^1_q(\Omega)}=1} \left(\|S(a)-S(a^\dag)\|_{H^1_\tau(\Omega)}\left\|\f{1}{S(a^\dag)} g \right\|_{H^1_{\tau^*}(\Omega)}\right)\  \forall a\in D(S).
\eeq
Then, applying \eqref{eq:g} to \eqref{eq:amina}, we conclude that \eqref{key:est} is valid.

(ii)
Let $\mathcal{P}_N$ be the PL decomposition as constructed in Example \ref{ex:PL} (ii). In view of \eqref{projection:N},    it holds that
\beq\label{eq:Fq}
F_q^t(\mathcal{P}_N)=H^{2t}_q(\Omega;\mathbb{C})  \quad \forall t\in [0,1/2] \quad \Rightarrow \quad F_q^{1/2}(\mathcal{P}_N)^*=(H^{1}_q(\Omega;\mathbb{C}))^*
\eeq
with equivalent norms.
Then,   \eqref{key:est} implies that \eqref{eq:cond} holds true for $T=S$, $Y=H_\tau^1(\Omega)$, $\Psi_0(\delta)= \delta$ and $\theta=1$. In conclusion, the claim (ii) follows from Theorem \ref{the:vsc}.

(iii)
Applying the interpolation inequality \eqref{interpolate}   with $s_1=s_2=1$, $\tau_1=\overline\tau$, $\tau_2=1$ and $\rho=\f{\overline\tau-\tau}{\tau (\overline\tau-1)}$   to  the right hand side of   \eqref{key:est} along with
\eqref{estover}, we obtain
\beq\label{key:est2}
\|a-a^\dag\|_{H_q^1(\Omega)^*}\leq CM_1^{1-\f{\overline\tau-\tau}{\tau (\overline\tau-1)}}\|S(a)-S(a^\dag)\|^{\f{\overline\tau-\tau}{\tau (\overline\tau-1)}}_{H^1_1(\Omega)}\quad \forall\, a\in D(S).
\eeq
 In view of \eqref{eq:Fq} and \eqref{key:est2},    we   see that \eqref{eq:cond} holds true for   $T=S$, $Y=H_1^1(\Omega)$, $\Psi_0(\delta)=\delta^ \f{\overline\tau-\tau}{\tau (\overline\tau-1)}$  and $\theta=1$. Thus,  by Theorem \ref{the:vsc},  the  claim (iii) is valid.

 (iv) Similarly, applying the interpolation inequality \eqref{interpolate}     with  $s_1=2,s_2=0,\tau_1=\tau_2=\tau$ and $\rho=1/2$
   to  the right hand side of    \eqref{key:est} together with
\eqref{estover2}, we have
\beq\label{key:est3}
\|a-a^\dag\|_{H_q^1(\Omega)^*}\leq CM_2^{\f 1 2}\|S(a)-S(a^\dag)\|^{\f 1 2}_{L^\tau (\Omega)}\quad \forall\, a\in D(S).
\eeq
In view of \eqref{eq:Fq} and \eqref{key:est3},   we   see that \eqref{eq:cond} holds true for $T=S$, $Y=L^\tau (\Omega)$, $\Psi_0(\delta)= \sqrt{\delta}$   and $\theta=1$.
In conclusion,  the claim (iv) follows from Theorem \ref{the:vsc}.
\end{proof}

The conditional stability estimate \eqref{key:est}  is the main key point to verify VSC  \eqref{VSCp} for $T=S$, as it implies the required condition \eqref{eq:cond} for Theorem \ref{the:vsc}.    Concluding from  Corollary \ref{corconv}, Theorem \ref{first}, and \eqref{psidecay}, we obtain the following convergence rates:
\begin{Corollary}\label{corfinal}
Under the assumptions of Theorem and  the parameter choice $\alpha(\delta):=\f{\delta^\ell}{\Psi(\delta)}$,  the  Tikhonov regularization method \eqref{mina} yields the convergence rates:
\begin{equation*}
\|a^\delta_{\alpha(\delta)}-a^\dag\|_p^{\hat p}=
\begin{cases}
O(\delta^{\f{\hat qs}{1+(\hat q-1)s}})\quad \text{as}\,\,\delta\to 0^+ \quad \textrm{in the case of \textup{(ii)} with $Y=H_\tau^1(\Omega)$}, \\
O(\delta^{\f{( \overline\tau-\tau) \hat qs}{\tau (\overline\tau-1)(1+(\hat q-1)s)}}
)\quad \text{as}\,\,\delta\to 0^+    \textrm{ in the case of \textup{(iii)} with $Y=H_1^1(\Omega)$}, \\
O(\delta^{\f{\hat q  s}{2+2(\hat q-1)s}}
)\quad \text{as}\,\,\delta\to 0^+  \quad \textrm{in the case of \textup{(iv)} with $Y=L^\tau(\Omega)$}.
\end{cases}
\end{equation*}
\end{Corollary}

As a conclusion,   different choices of  $Y$-norms  and   estimates \eqref{key:est}-\eqref{estover2} lead to different convergence rates.  The case (iv) with the weakest norm $Y=L^\tau(\Omega)$ is mostly relevant for applications since
measurement   in higher Sobolev norms is typically difficult to realize in practice.
{
\begin{Remark}
In the case of  \textup{(iv)} of Theorem \eqref{first} with $p=2$, one may expect convergence rate   $\|a^\delta_{\alpha(\delta)}-a^\dag\|_2^2=O(\delta^{\frac{2s}{s+2}})$. Thus, our convergence rates is suboptimal in this case, which may be caused by the treatment of the nonlinearity
of the forward operator.  It is an open question that we shall investigate as a future goal. 
\end{Remark}
}

\subsection{Discussion of hypotheses  in Theorem \ref{first}}\label{sec:hyp}

In the following, we discuss the assumptions (b), \eqref{estover}
and \eqref{estover2} with more details. Although  $S(a)$ belongs only to
$H^1_\tau(\Omega)$ for all   $1\leq \tau<\f{n}{n-1}$, it turns out that the difference $S(a)-S(b)$ for all $a,b\in D(S)$ enjoys a better regularity property, provided that $\kappa$ is regular enough.  This fact allows us to verify the assumptions (b),  \eqref{estover}
and \eqref{estover2} under the following material assumption:

\begin{enumerate}

\item[{\bf (A) }]
There exist $C^{1}$ domains $\Omega_j\subset \R^n$, $j=1,\ldots, N$ such that
$$
\overline{\Omega}_i\cap \overline{\Omega}_j=\emptyset \quad \forall\, i\neq j\,\,\text{and} \,\,
\overline{\Omega}_j\subset \Omega.
$$
Furthermore,     it holds that
$$\kappa\mid_{_{\Omega_c}}\in  C(\overline{\Omega}_c)\quad\text{and}\quad
\kappa\mid_{_{\Omega_j}}\in  C(\overline{\Omega}_j)\quad \forall\, j=0,1\ldots, N,
$$
where $\Omega_c:=\Omega\backslash \bigcup_{j=1}^N$.

\end{enumerate}

\begin{Remark}   \normalfont
 To model a heterogeneous medium, the assumption of piecewise continuous material functions  is reasonable and often used in the mathematical study of elliptic equations (cf. \cite{elschner2007optimal}).

\end{Remark}

\begin{Lemma}[{Theorem 1.1, Remarks 3.17--3.19  in \cite{elschner2007optimal}}]\label{well:neu}
 Assume that {\bf (A)} holds true and $1<r,\tau<+\infty$ such that
 \beq\label{eq:tau2}
 \begin{cases}
\tau\in (1,+\infty) \quad &\text{if}\, \, r\geq n,\\
 \tau \in (1,\f{nr}{n-r}] \quad  &\text{if}\,\, 1<r<n.
 \end{cases}
 \eeq
 Then, for every $f\in L^r(\Omega)$,   the homogeneous
Neumann problem
 \beq\label{neu}
 \begin{cases}
 -\nabla\cdot \kappa\nabla u+u= f \quad &\text{in}\,\, \Omega,\\
 \kappa\frac{\p u}{\p \nu}=0 \quad &\text{on}\,\, \Gamma
 \end{cases}
 \eeq
 admits  a unique weak solution $u \in H^1_\tau(\Omega)$ satisfying
\beq\label{est:neu}
\| u\|_{H^1_\tau(\Omega)}\lesssim \|f\|_{L^r(\Omega)}.
\eeq

\end{Lemma}

\begin{Remark}    \normalfont
As a special case of \cite{elschner2007optimal} and an analogue of  \cite{disser2015optimal}  for Neumann conditions,  the material assumption ${\bf (A)}$ implies that  for every $1<\tau<\infty$ and $\tau^*=\f{\tau}{\tau-1}$
the operator
$
-\nabla \cdot \kappa\nabla+1: H^1_\tau(\Omega)\to H^1_{\tau^*}(\Omega)^*
$	
is  a topological isomorphism.  We note that   \eqref{eq:tau2} implies
\beq\label{ineq:rn}
1-\f{1}{r}\geq 1-\f{1}{\tau}-\f{1}{n} \quad \Rightarrow \quad  \frac{1}{r^*}\geq \frac{1}{\tau^*}-\frac{1}{n}.
\eeq
In view of \eqref{ineq:rn},    Proposition \ref{pro:sobolev} yields the continuous embedding
$
\ H^1_{\tau^*}(\Omega)\embed L^{r^*}(\Omega).
$
  Therefore, under ${\bf (A)}$ and \eqref{eq:tau2},    \eqref{neu} admits  for every $f\in L^r(\Omega) \embed  H^1_{\tau^*}(\Omega)^* $ a unique weak solution $u\in H^1_\tau(\Omega)$ with $\tau$ as in \eqref{eq:tau2}. This unique weak solution satisfies
$$
\| u\|_{H^1_\tau(\Omega)}\lesssim\|f\|_{H^1_\tau(\Omega)^*}\lesssim \|f\|_{L^r(\Omega)}.
$$
Let us also mention  that ${\bf (A) }$  cannot be relaxed due to
the counterexamples in \cite{elschner2007optimal}.
\end{Remark}

\begin{Lemma}\label{lemma:prior}
Assume that ${\bf (A)}$ holds and $\frak p > \frac{n}{2}$.

\begin{enumerate}

\item[\textup{(i)}]  If $n=2$, then for every
$$
\begin{cases}
\tau\in (1,+\infty) \,\, &if \,  \pf> 2, \\
\tau\in (1,\f{2\pf}{2-\pf})\,\, &if \, 1<\pf \le 2
\end{cases}	
$$
there exists a constant $C>0$ such that
$$
\|S(a)-S(b)\|_{H^1_\tau(\Omega)}\leq C\quad \forall a,b\in D(S).
$$

\item[\textup{(ii)}]  If $n=3$, then for every
$\tau\in (1,\pf)$  there exists a constant $C>0$ such that
$$
\|S(a)-S(b)\|_{H^1_\tau(\Omega)}\leq C\quad \forall a,b\in D(S).
$$

\end{enumerate}
\end{Lemma}
\begin{proof}
According to  Definition    \ref{def:ws},  we have
\begin{equation}\label{apriori1}
\begin{aligned}
&\int_{\Omega} \kappa\nabla(S(a)-S(b))\cdot \nabla \varphi
+(S(a)-S(b))\varphi dx\\
&=\int_{\Omega} (S(a)-S(b)+bS(b)-aS(a))\varphi dx \quad\forall \varphi\in C^\infty(\overline{\Omega}) \quad \forall a,b \in D(S).
\end{aligned}
\end{equation}
Let us first consider the case $n=2$. In view of Theorem \ref{the:well} and   Proposition \ref{pro:sobolev} (i), it holds that
$$
\|S(a)\|_{s} \le C(s) \quad \forall a \in D(S) \quad  \forall 1\le s<\infty,
$$
for some  constant $C(s)>0$, independent of $a \in D(S)$. For this reason, making use of the definition $D(S) \subset L^\pf(\Omega)$ (see \eqref{def:ds}), it follows that
$$
\|aS(a) \|_{r}  \le C(r) \quad \forall a \in D(S) \quad \forall 1\le r <  \pf.
$$
for some constant $C(r)>0$, independent of $a \in D(S)$. Combining the above two inequalities yields that
\beq\label{bd:2}
\|S(a)-S(b)+bS(b)-aS(a)\|_{r}\leq C(r) \quad \forall a,b \in D(S) \quad \forall 1\le r <  \pf.
\eeq
If $\pf >2$, then we may choose $r = 2=n$ in \eqref{bd:2} such that applying   Lemma \ref{well:neu} to \eqref{apriori1} yields the claim (i) for   $\tau\in (1,+\infty)$ and $\frak p > 2$.  If $1< \pf \le 2=n$,  then for every    $\tau\in (1,\f{2\pf}{2-\pf})$, we can    find  an  $r<\pf\le n$ such that
$\tau<\f{2r}{2-r} = \f{nr}{n-r} $. Therefore, in view of  \eqref{bd:2},    applying again Lemma \ref{well:neu} to \eqref{apriori1} yields   the claim (i)     for $\tau\in (1,\f{2\pf}{2-\pf})$ and  $1< \frak p \le 2$.

Now let us consider the case $n=3$ and $\frak p > \frac{3}{2}$. Theorem  \ref{the:well} and  Proposition \ref{pro:sobolev} (i)  ensure that
\begin{equation}\label{apriori2}
\|S(a)\|_{s} < C(s)  \quad \forall a \in D(S) \quad\forall  1\le s<3,
\end{equation}
for some constant $C(s)>0$, independent of $a \in D(S).$
Then, making use of the definition of  $D(S) \subset L^\pf(\Omega)$ (see \eqref{def:ds}),   the generalized H\"{o}lder inequality implies       that
\begin{equation}\label{apriori3}
\|aS(a)\|_{r} \le \|a\|_{\pf}\|S(a)\|_{{\frac{r\frak p}{\frak p-r}}} \le MC(r,\frak p)   \quad   \forall a \in D(S) \quad\forall1 \le r <  \f{3\pf}{3+\pf},
\end{equation}
 where we have used \eqref{apriori2} since $1\le  \frac{r\frak p}{\frak p-r} <3$ holds true for  all $1 \le r <  \f{3\pf}{3+\pf}$. Altogether, since $\f{3\pf}{3+\pf}<3$,  \eqref{apriori2} and \eqref{apriori3} yield
 \beq\label{bd:3}
\|S(a)-S(b)+bS(b)-aS(a)\|_{r}\leq C(r,\frak p)   \quad   \forall a,b \in D(S) \quad \forall 1 \le r <   \f{3\pf}{3+\pf} < 3=n,
\eeq
for some constant $C(r,\frak p)>0$, independent of $a,b \in D(S)$. In view of \eqref{bd:3}, applying Lemma \ref{well:neu} to \eqref{apriori1}, we come to the conclusion that    for every
$\tau\in (1,\pf)$,  there exists a constant $C>0$ such that
$$
\|S(a)-S(b)\|_{H^1_\tau(\Omega)}\leq C\quad \forall a,b\in D(S).
$$
This completes the proof.
\end{proof}

\begin{Lemma}\label{lemma:prior2}
Let $n\in \{2,3\}$ and $\pf>\f{n}{2}$.  Assume  that $\kappa\in C^{0,1}(\overline{\Omega})$.  Then, for every $\tau\in (1,\overline{\tau})$ with
\beq\label{sn}
\overline{\tau}:=
\begin{cases}
\pf\quad &n=2,\\
\f{3\pf}{3+\pf}\quad &n=3,
\end{cases}
\eeq
there exists a constant $C>0$ such that
\beq\label{bd:2tau}
\|S(a)-S(b)\|_{H^2_\tau (\Omega)}\le C \quad \forall a,b\in  D(S).
\eeq
%
\end{Lemma}
\begin{proof}
Let $a,b\in D(S)$. By the definition of the weak solution \eqref{def:ws},
\begin{align*}
&\int_{\Omega} \kappa\nabla(S(a)-S(b))\cdot \nabla \varphi
+(S(a)-S(b))\varphi dx\\
&=\int_{\Omega} (S(a)-S(b)+bS(b)-aS(a))\varphi dx \quad\forall \varphi\in C^\infty(\overline{\Omega}).
\end{align*}
Therefore, in view of \eqref{bd:2} (for $n=2$) and  \eqref{bd:3} (for $n=3$),  the classical $W^{2,\tau}(\Omega)$-regularity result for elliptic equations (\cite[Theorem 2.4.1.3]{grisvard2011elliptic}) implies  \eqref{bd:2tau}.
\end{proof}

In conclusion,  we see that assumptions  \eqref{estover}
and \eqref{estover2} can be guaranteed by Lemma \ref{lemma:prior} and
 Lemma \ref{lemma:prior2}, respectively.

\section{Conclusion}  Based on  the Littlewood-Paley theory  and the  concept of    $\mathcal{R}$-boundedness, we   developed    sufficient criteria (Theorem \ref{the:vsc}) for VSC  \eqref{VSCp} in $L^p(\Omega,\mu)$-spaces with $1<p<+\infty$. The proposed sufficient criteria consist of the existence of an LP-decomposition for the complex dual space $L^q(\Omega,\mu;\C)$ ($q=\frac{p}{p-1}$) together with a conditional stability estimate
and a regularity requirement for the true solution  in terms of   Triebel-Lizorkin-type  norms. In Section \ref{sec:app}, the developed abstract result is applied to
 the  inverse reconstruction problem  of    unbounded diffusion $L^p(\Omega)$-coefficients in elliptic equations with measure data \eqref{elliptic}.  We derived existence and plain convergence results for the associated Tikhonov regularization problem \eqref{mina} with $L^p(\Omega)$-norm penalties (Theorem \ref{tikop}). As final results (Theorem \ref{first} and Lemmas \ref{lemma:prior} and \ref{lemma:prior2})), we prove that all requirements of   Theorem \ref{the:vsc} are satisfied for the inverse problem \eqref{eq:ip}, leading to convergence rates for the Tikhonov regularization method \eqref{mina} (Corollary  \ref{corfinal}).

 Our future goals are threefold. First, noticing that there are recent progresses  on VSC for $\ell^1$-regularization (see. e.g. \cite{AHR14,flemming2018,Wang}), we aim at extending our study to the Tikhonov regularization method with $L^1(\Omega,\mu)$-penalties. On the other hand, in some applications, the unknown solution could fail to have
a finite penalty value, if the penalty is oversmoothing. Recently, such  oversmoothing regularizations have been studied for  inverse problems in
Hilbert scales (see e.g. \cite{Egger19,Hofmann20,HofMat18}). As  Triebel-Lizorkin-type  space allows us to  work with      scales of Banach space through the use of sectorial operators,  it would be attempting to study
 oversmoothing regularizations under an $L^p(\Omega)$-setting.  Thirdly,
 we would like to extend the developed results to  nonlinear and non-smooth PDEs, in particular for those arising from electromagnetic applications (\cite{LamYousept2020,Yousept2012,Yousept2013,Yousept2017}).

{
\section{Appendix} 

\subsection{Properties of $\varphi_p$ from Remark \ref{concremark}} It is obvious that  $\varphi_p(\mathbb{R})\subset \mathbb{R}$. Let $N\in \mathbb{Z}$ and 
$0<\tau\leq 1$ such that $N+\tau=p-1$.  Since 
$\varphi_p(t)=t^{p-1}$ for $t>0$, and $\varphi_p(t)=-(-t)^{p-1}$ for $t<0$, it follows   for all 
$l=0,1,2\cdots,N$ that  
\beq\label{varphip}
\varphi_p^{(l)}(0)=0\,\, \text{and}\,\, \varphi_{p}^{(l)}(t)=\begin{cases}
c_l t^{p-1-l}   &\text{if }t>0,\\
(-1)^{l+1}c_l (-t)^{p-1-l}  &\text{if }t<0,
\end{cases}
\eeq
where $c_l:=(p-1)\cdot (p-2)\cdots (p-l)$,
 and hence $|\varphi_p^{(l)}(t)|\leq  c_l |t|^{p-1-l}$ for all $t\in \R$ with $l=0,1,\cdots, N$. Using \eqref{varphip} and 
the inequality: 
$(t+s)^\tau \leq t^\tau +s^\tau $ for all $t,s\geq 0$,
  we have that 
\[
\frac{|\varphi^{(N)}_{p}(t_1)-\varphi^{(N)}_{p}(t_0) |}{|t_1-t_0|^{\tau}}
=c_l\frac{|t_1^{\tau}-t_0^\tau |}{|t_1-t_0|^{\tau}}{=c_l\frac{(t_1-t_0+t_0)^{\tau}-t_0^\tau }{(t_1-t_0)^{\tau}}\leq c_l},
\]
whenever $t_1>t_0\geq 0$. The case $t_0>t_1\geq 0$ can be handled similarly. Thus,
\beq\label{varphip1}
\sup_{t_1,t_0\geq 0,t_0\neq t_1}\frac{|\varphi^{(N)}_{p}(t_1)-\varphi^{(N)}_{p}(t_0) |}{|t_1-t_0|^{\tau}}\leq c_l.
\eeq
In view of \eqref{varphip1} and the fact that $\varphi^{(N)}_{p}(t)=(-1)^{l+1}\varphi^{(N)}_{p}(-t)$ for any $t<0$ by  \eqref{varphip}, we obtain that 
\beq\label{varphip2}
\sup_{t_1,t_0\leq 0, t_0\neq t_1}\frac{|\varphi^{(N)}_{p}(t_1)-\varphi^{(N)}_{p}(t_0) |}{|t_1-t_0|^{\tau}}\leq c_l.
\eeq
If $t_1t_0<0$, then from \eqref{varphip} it follows that 
\beq\label{varphip3}
\frac{|\varphi^{(N)}_{p}(t_1)-\varphi^{(N)}_{p}(t_0) |}{|t_1-t_0|^{\tau}}\leq 
\frac{|\varphi^{(N)}_{p}(t_1)|}{|t_1-t_0|^{\tau}}+\frac{|\varphi^{(N)}_{p}(t_0)|}{|t_1-t_0|^{\tau}}
\leq \frac{|\varphi^{(N)}_{p}(t_1)|}{|t_1|^{\tau}}+\frac{|\varphi^{(N)}_{p}(t_0) |}{|t_0|^{\tau}}
\leq 2 c_l.
\eeq
Therefore, we   conclude from \eqref{varphip1}--\eqref{varphip3} that 
$$
\sup_{t_0\neq t_1}\frac{|\varphi^{(N)}_{p}(t_1)-\varphi^{(N)}_{p}(t_0) |}{|t_1-t_0|^{\tau}}\leq 2c_l.
$$ 

\subsection{Proof of \eqref{varphplus3} } 
For $s\in \R$, $1\leq q< \infty$, and $1\leq r\leq \infty$, let $F^s_{q,r}(\mathbb{R}^n;\mathbb{C})$ denote the Triebel-Lizorkin space as defined in \cite[Page 8]{Runst}, and $F^s_{q,r}(\mathbb{R}^n)$ the subspace of $F^s_{q,r}(\mathbb{R}^n;\mathbb{C})$ consisting of real-valued functions in $F^s_{q,r}(\mathbb{R}^n;\mathbb{C})$, which satisfies 
\[
F^s_{q,2}(\mathbb{R}^n)=H^s_{q}(\mathbb{R}^n)
\]
if $s\in \R$ and $1<q<\infty$
(see  \cite[Section 2.1.2. Proposition 1 (vii)]{Runst}).
In addition,  the inclusion 
\beq\label{Hsinclusion}
H^{s'}_{q}(\mathbb{R}^n)\subset F^{s''}_{q,r''}(\mathbb{R}^n)
\eeq
is valid if $1\leq q<\infty$, $s'>s''$, and $1\leq r',r''\leq \infty$ (see  \cite[Section 2.2.1, Proposition 1]{Runst}). From \cite[Section 5.4.4, Theorem 1]{Runst}, it follows that the inclusion
$$
T_{\varphi_{p,+}}(F^{s/(p-1)}_{p,2(p-1)}(\R^n))\subset F^{s}_{q,2}(\R^n)
=H^s_{q}(\mathbb{R}^n)
$$
holds if $1<p<2$ and $0<s<p-1$, which, together with \eqref{Hsinclusion}, implies the inclusion 
\beq\label{inclusion22}
T_{\varphi_{p,+}}(H^t_p(\R^n))\subset 
H^s_{q}(\mathbb{R}^n),
\eeq
if $t>\f{s}{p-1}$. By the argument used in  Remark \ref{concremark} (i), we   deduce that 
 \eqref{inclusion22} remains true if we replace $H^t_p(\R^n)$ and $H^s_{q}(\mathbb{R}^n)$ by 
 $H^t_p(\Omega)$ and $H^s_{q}(\Omega)$, respectively. This completes the proof.

}


\footnotesize
\begin{thebibliography}{10}


\bibitem{alibert1997boundary}
{\sc J.-J. Alibert and J.-P. Raymond}, {\em Boundary control of semilinear
  elliptic equations with discontinuous leading coefficients and unbounded
  controls}, Numerical Functional Analysis and Optimization, 18 (1997),
  pp.~235--250.

\bibitem{AHR13}
{\sc S.~W.~Anzengruber,  B.~Hofmann and R.~Ramlau},
 {\em On the interplay of basis smoothness and specific range conditions occurring in sparsity regularization},
 {Inverse Problems}, 29 (2013), p.~125002.


\bibitem{AHR14}
{\sc S.~W.~Anzengruber,  B.~Hofmann, and  P.~Math\'e}, {\em Regularization properties
of the sequential discrepancy principle for Tikhonov regularization in Banach
spaces}, Appl. Anal. 93 (2014), pp.~1382--1400.


\bibitem{Barbu}
{\sc V.~Barbu}, {\em Nonlinear differential equations of monotone types in Banach spaces.} Springer Science \& Business Media.


{
\bibitem{BMW2008}
{\sc G.~Bourdud, M.~Moussai and W. Sickel}, {\em Towards sharp superposition theorems in Besov and Lizorkin--Triebel spaces}, Nonlinear Analysis, Function Spaces and Applications,  259 (2008), 
pp.~2289-2912.

\bibitem{BMW2010}
{\sc G.~Bourdud, M.~Moussai and W. Sickel}, {\em Composition operators on Lizorkin-Triebel spaces}, J. of Funct. Anal 259 (2010), 
pp.~1098-1128.
}

\bibitem{boct2010extension}
{\sc R.~I. Bo{\c{t}} and B.~Hofmann}, {\em An extension of the variational
  inequality approach for obtaining convergence rates in regularization of
  nonlinear ill-posed problems}, The Journal of Integral Equations and
  Applications,  (2010), pp.~369--392.

\bibitem{bourgain1986vector}
{\sc J.~Bourgain}, {\em Vector-valued singular integrals and the $h^1$-bmo
  duality}, Probability theory and harmonic analysis,  (1986), pp.~1--19.


\bibitem{BFH13}
{\sc M.~Burger, J.~Flemming and B.~Hofmann},
{\em Convergence rates in $\ell^1$-regularization if the sparsity assumption fails.}
 {Inverse Problems}, 29 (2013), p.~025013.


\bibitem{Chavent94}
{\sc G.~Chavent and K.~Kunisch}, {\em Convergence of Tikhonov regularization for constrained ill-posed inverse problems,} Inverse Problems, 10 (1994), pp.~63-76.



\bibitem{chen2018variational}
{\sc D.-H. Chen and I.~Yousept}, {\em Variational source condition for
  ill-posed backward nonlinear  Maxwell's equations}, Inverse Problems, 35
  (2018), p.~025001.

\bibitem{cowling1996banach}
{\sc M.~Cowling, I.~Doust, A.~Micintosh, and A.~Yagi}, {\em Banach space
  operators with a bounded $H^\infty$-functional calculus}, journal of the
  australian mathematical society, 60 (1996), pp.~51--89.

\bibitem{disser2015optimal}
{\sc K.~Disser, H.-C. Kaiser, and J.~Rehberg}, {\em Optimal sobolev regularity
  for linear second-order divergence elliptic operators occurring in real-world
  problems}, SIAM Journal on Mathematical Analysis, 47 (2015), pp.~1719--1746.


\bibitem{Egger19}
{\sc H.~Egger and B.~Hofmann}, {\em Tikhonov regularization in Hilbert scales under conditional
stability assumptions}, Inverse Problems 34 (2018), pp.~115015.

\bibitem{EHN96}
{\sc  H.~W.~Engl, M.~Hanke,  A.~Neubauer},  {\em Regularization of inverse problems}, volume 375 of Mathematics and its Applications, Kluwer Academic Publishers Group, Dordrecht, 1996.

\bibitem{Eng89}
{\sc H.W.~Engl, K.~Kunisch and A.~Neubauer}, {\em Convergence rates for Tikhonov regularization of
nonlinear ill-posed problems}, Inverse problems, 5 (1989), pp.~523--540.

\bibitem{EZ00}
{\sc H.W.~Engl and J.~Zou}, {\em A new approach to convergence rate analysis of Tikhonov
regularization for parameter identification in heat conduction}, Inverse Problems
16 (2000),pp.~1907--1923.


\bibitem{elschner2007optimal}
{\sc J.~Elschner, J.~Rehberg, and G.~Schmidt}, {\em Optimal regularity for
  elliptic transmission problems including $C^1$ interfaces}, Interfaces and
  Free Boundaries, 9 (2007), pp.~233--252.



\bibitem{FJZ12}
{\sc H.~Feng, D.~Jiang, J.~Zou}, {\em Convergence rates of Tikhonov regularizations for parameter identification in
a parabolic-elliptic system.} Inverse Problems 28 (2012), p.~104002.

\bibitem{flemming2010theory}
{\sc J.~Flemming}, {\em Theory and examples of variational regularization with
  non-metric fitting functionals}, Journal of Inverse and Ill-Posed Problems,
  18 (2010), pp.~677--699.



 \bibitem{flemming2018}
{\sc J. Flemming, D. Gerth},   {\em Injectivity and weak$^*$-to-weak continuity
suffice for convergence rates in $\ell^1$-regularization.}
J. Inverse Ill-Posed Probl. 26 (2018), pp.~85--94.



\bibitem{flemming2018theory2}
{\sc J.~Flemming}, {\em Existence of variational source conditions for nonlinear
inverse problems in Banach spaces.} J. Inverse Ill-Posed Probl. 26 (2018), pp.~277--286.




 \bibitem{Gilbarg}
 {\sc D.~Gilbarg, N.S.~Trudinger}, {\em Elliptic Partial Differential Equations of Second Order}, Springer, Berlin, Heidelberg, New York 1977.
  

\bibitem{grasmair2010generalized}
{\sc M.~Grasmair}, {\em Generalized  Bregman distances and convergence rates for
  non-convex regularization methods}, Inverse Problems, 26 (2010), p.~115014.

\bibitem{grisvard2011elliptic}
{\sc P.~Grisvard}, {\em Elliptic problems in nonsmooth domains}, SIAM, 2011.



\bibitem{Hasse}
{\sc M.~Haase}, {\em The Functional Calculus for Sectorial Operators}, Operator
  Theory: Advances and Applications, Birkh{\"a}user Basel, 2006.

\bibitem{hofmann2007convergence}
{\sc B.~Hofmann, B.~Kaltenbacher, C.~Poeschl, and O.~Scherzer}, {\em A
  convergence rates result for Tikhonov regularization in Banach spaces with
  non-smooth operators}, Inverse Problems, 23 (2007), p.~987.

\bibitem{Hofmann20}
{\sc B.~Hofmann and R.~Plato}, {\em Convergence results and low order rates for
nonlinear Tikhonov regularization with oversmoothing penalty term},
Electronic Transactions on Numerical Analysis, 53 (2020), pp. 313-328.

\bibitem{HofMat12}
{\sc B.~Hofmann and P.~Math\'e,} {\em Parameter choice in Banach space regularization under variational
  inequalities,}
 {Inverse Problems}, 28 (2012), p.~104006.


\bibitem{HofMat18}
{\sc B.~Hofmann and P.~Math\'e,} {\em Tikhonov regularization with oversmoothing
penalty for non-linear ill-posed problems in Hilbert scales},
Inverse Problems, 34 (2018), pp.~015007.


\bibitem{hohage2019optimal}
{\sc T.~Hohage and P.~Miller}, {\em Optimal convergence rates for sparsity
  promoting wavelet-regularization in Besov spaces}, Inverse Problems, 35
  (2019), p.~065005.

\bibitem{hohage2015verification}
{\sc T.~Hohage and F.~Weidling}, {\em Verification of a variational source
  condition for acoustic inverse medium scattering problems}, Inverse Problems,
  31 (2015), p.~075006.

\bibitem{hohage2017characterizations}
\leavevmode\vrule height 2pt depth -1.6pt width 23pt, {\em Characterizations of
  variational source conditions, converse results, and maxisets of spectral
  regularization methods}, SIAM Journal on Numerical Analysis, 55 (2017),
  pp.~598--620.



\bibitem{hytonen2018analysis}
{\sc T.~Hyt{\"o}nen, J.~Van~Neerven, M.~Veraar, and L.~Weis}, {\em Analysis in
  Banach Spaces: Volume II: Probabilistic Methods and Operator Theory},
  vol.~67, Springer, 2018.






\bibitem{Kaltenbacher08}
{\sc  B.~Kaltenbacher}, {\em A note on logarithmic convergence rates for nonlinear Tikhonov regularization}, \newblock{Journal of Inverse and Ill-Posed Problems}, 16 (2008), pp.~79-88.

\bibitem{Kaltenbacher09}
{\sc B.~Kaltenbacher, F.~Sch\"{o}pfer and T.~Schuster}, {\em Iterative methods for nonlinear ill-posed problems in Banach spaces: convergence and applications to parameter identification problems}, Inverse Problems, 25 (2009), p.~065003.


\bibitem{kriegler2016paley}
{\sc C.~Kriegler and L.~Weis}, {\em Paley--Littlewood decomposition for
  sectorial operators and interpolation spaces}, Mathematische Nachrichten, 289
  (2016), pp.~1488--1525.

\bibitem{krivine1978constantes}
{\sc J.-L. Krivine}, {\em Constantes de grothendieck et fonctions de type
  positif sur les spheres}, S{\'e}minaire Analyse fonctionnelle (dit"
  Maurey-Schwartz"),  (1978), pp.~1--17.

\bibitem{Kunstmann}
{\sc P.~Kunstmann and A.~Ullmann}, {\em $\mathcal{R}_s$-sectorial operators and
  generalized triebel-lizorkin spaces}, J. Fourier Anal. Appl., 20 (2014),
  pp.~135--185.

   \bibitem{LamYousept2020}
{\sc K.F.~Lam and I.~Yousept}, {\em  Consistency of a phase field
regularisation for an inverse problem
governed by a quasilinear Maxwell
system},
Inverse Problems 36 (2020) 045011



\bibitem{Le}
{\sc D.~Le and H. Smith}, {\em Strong positivity of solutions to parabolic and elliptic equations on nonsmooth domains}, J. Math. Anal. Appl. 275 (2002) 208--221.


 \bibitem{Lechleiter08}
{\sc A.~Lechleiter and A.~Rieder}, {\em Newton regularizations for impedance tomography: convergence by local injectivity}, Inverse Problems, 24 (2008),p.~065009.

 
 \bibitem{Natterer}
{\sc F.~Natterer}, {\em Error bounds for Tikhonov regularization in Hilbert scales}, Applicable Anal. 18 (1984), 29--37.



\bibitem{ouhabaz2009analysis}
{\sc E.-M. Ouhabaz}, {\em Analysis of heat equations on domains.(LMS-31)},
  Princeton University Press, 2009.


\bibitem{pruss2016moving}
{\sc J.~Pr{\"u}ss and G.~Simonett}, {\em Moving interfaces and quasilinear
  parabolic evolution equations}, vol.~105, Springer, 2016.


  {
  
  \bibitem{Runst}
{\sc T.~Runst and W.~Sickel}, {\em Sobolev spaces of fractional order,
Nemytskij operators, and nonlinear partial differential equations.} de
Gruyter Series in Nonlinear Analysis and Applications, no. 3. Walter
de Gruyter \& Co., Berlin, 1996.
  }

\bibitem{SEK93}
{\sc O.~Scherzer, H.~W.~Engl, and K.~Kunisch}, {\em Optimal a Posteriori Parameter Choice for Tikhonov Regularization for Solving Nonlinear Ill-Posed Problems},  SIAM J. Numer. Anal., 30 (1993), pp.~1796--1838.



\bibitem{schuster2012regularization}
{\sc T.~Schuster, B.~Kaltenbacher, B.~Hofmann, and K.~S. Kazimierski}, {\em
  Regularization methods in Banach spaces}, vol.~10, Walter de Gruyter, 2012.


{
\bibitem{Sickel}
{\sc W.~Sickel}, {\em Superposition of functions in Sobolev spaces of fractional order A
survey, in: Partial Differential Equations}, Banach Center Publications 27, Polish
Academy of Sciences, Warsaw, 1992, pp.  481--497.
}
%
%

\bibitem{stein1993harmonic}
{\sc E.~M. Stein and T.~S. Murphy}, {\em Harmonic analysis: real-variable
  methods, orthogonality, and oscillatory integrals}, vol.~3, Princeton
  University Press, 1993.

\bibitem{taylor2007tools}
{\sc M.~E. Taylor}, {\em Tools for PDE: pseudodifferential operators,
  paradifferential operators, and layer potentials}, no.~81, American
  Mathematical Soc., 2007.

 
  \bibitem{Tautenhahn}
{\sc U.~Tautenhahn}, {\em Error estimates for regularization methods in Hilbert scaless} SIAM J. Numer. Anal. 33 (1996), no. 6, pp. 2120--2130


\bibitem{triebel1995interpolation}
{\sc H.~Triebel}, {\em Interpolation theory, function spaces, differential
  operators}, Johann Ambrosius Barth Verlag, 1995.


\bibitem{Wang}
{\sc W. Wang, S. Lu, B. Hofmann and  J. Cheng}, {\em Tikhonov regularization with
$\ell^0$-term complementing a convex penalty: $\ell^1$-convergence under
sparsity constraints}, J. Inverse Ill-Posed Probl. 27 (2019),  pp. 575--590.


\bibitem{weidling2015variational}
{\sc F.~Weidling and T.~Hohage}, {\em Variational source conditions and
  stability estimates for inverse electromagnetic medium scattering problems},
  arXiv preprint arXiv:1512.06586,  (2015).




 \bibitem{hohage2020}
 {\sc F.~Weidling, B.~Sprung, and T.~Hohage,}   {\em Optimal convergence rates for Tikhonov regularization in Besov spaces}. SIAM Journal on Numerical Analysis, 58 (2020), pp.~21--47.


\bibitem{xu1991inequalities}
{\sc H.-K. Xu}, {\em Inequalities in  Banach spaces with applications},
  Nonlinear Analysis: Theory, Methods \& Applications, 16 (1991),
  pp.~1127--1138.

\bibitem{Yagi}
{\sc A.~Yagi}, {\em Abstract parabolic evolution equations and their
  applications}, Springer Science \& Business Media, 2009.


 \bibitem{Yousept2012} I. Yousept.  Finite element analysis of an optimal control problem in the coefficients of time-harmonic eddy current equations. J. Optim. Theory Appl.,   154  (2012) 879--903.


\bibitem{Yousept2013}
I. Yousept, \textit{Optimal control of quasilinear $\boldsymbol{H}(\mathbf{curl})$-elliptic partial differential equations in
magnetostatic field problems}, SIAM J. Control and Optim., 51 (2013), 3624-3651.



\bibitem{Yousept2017}
I.~Yousept.
\newblock \textit{Optimal control of non-smooth hyperbolic evolution
 Maxwell equations in type-{II} superconductivity},
\newblock {{SIAM J. Control and Optim.}},  55 (2017), 2305-2332.





  \end{thebibliography}
 \end{document}